\documentclass[sn-mathphys,Numbered]{sn-jnl}

\usepackage{graphicx}%
\usepackage{multirow}%
\usepackage{amsmath,amssymb,amsfonts}%
\usepackage{amsthm}%
\usepackage{mathrsfs}%
\usepackage[title]{appendix}%
\usepackage{xcolor}%
\usepackage{textcomp}%
\usepackage{manyfoot}%
\usepackage{booktabs}%
\usepackage{algorithm}%
\usepackage{algorithmicx}%
\usepackage{algpseudocode}%
\usepackage{listings}%
\usepackage{comment}
\usepackage{bbm}
\usepackage{diagbox}

\newcommand{\rset}{\mathbb{R}}

\newcommand{\Hhmax}{H_{\phi,\max} }

\newcommand{\inas}{\overset{\text{a.s.}}{\to}}
 \newcommand{\red}{\textcolor{black}}

\newtheorem{theorem}{Theorem}
\newtheorem{definition}{Definition}%
\newtheorem{assumption}{Assumption}
\newtheorem{lemma}{Lemma}
\newtheorem{corollary}{Corollary}

\begin{document}

\title[Coordinate descent methods beyond smoothness and separability]{Coordinate descent methods beyond smoothness and separability}


\author[1]{\fnm{Flavia} \sur{Chorobura}}\email{flavia.chorobura@stud.acs.upb.ro}

\author*[1,2]{\fnm{Ion} \sur{Necoara}}\email{ion.necoara@upb.ro}


\affil*[1]{\orgdiv{Automatic Control and Systems
	Engineering Department}, \orgname{University Politehnica Bucharest}, \orgaddress{\street{Splaiul Independentei, 313}, \city{Bucharest}, \postcode{060042},  \country{Romania}}}

\affil[2]{\orgdiv{} \orgname{ Gheorghe Mihoc-Caius Iacob  Institute of Mathematical Statistics and Applied Mathematics of the Romanian Academy}, \orgaddress{\city{Bucharest}, \postcode{10587}, \country{Romania}}}



\abstract{This paper deals with convex nonsmooth optimization problems. We introduce a general smooth approximation framework for the original function and apply random (accelerated) coordinate descent methods for minimizing the corresponding smooth approximations. Our framework covers the most important classes of smoothing techniques from the literature. Based on this general framework for  the smooth approximation and using coordinate descent type methods we derive convergence rates in function values for the original objective.  Moreover, if the original function satisfies a growth condition, then we prove that the smooth approximations also inherits this condition and consequently the convergence  rates are improved in this case. We also present a relative randomized coordinate descent algorithm for solving nonseparable minimization problems with the objective function  relative smooth along coordinates w.r.t. a (possibly nonseparable) differentiable function. For this algorithm we also derive
 convergence rates in the convex case and under the growth condition for the objective.}

\keywords{Convex optimization, growth condition, nonsmooth and nonseparable objective, coordinate descent, convergence analysis.}



\maketitle
\section{Introduction}\label{sec1}
\noindent In this paper we study (accelerated) random (block) coordinate descent methods for solving nonsmooth convex optimization problems of the form: 
\vspace{-0.1cm}
\begin{align}
\label{eq:prob}
F^* = \min_{x \in \mathbb{R}^n} F(x),
\vspace{-0.4cm}
\end{align}
where $F: \mathbb{R}^n \to \bar{\mathbb{R}}$ is a general proper closed convex function (possibly nonseparable). First order methods for nonsmooth optimization, e.g. subgradient type methods, have convergence rates of order $\mathcal{O}(1/\epsilon^2)$, where $\epsilon$ is the desired  accuracy for the approximate solution. In order to improve convergence, smooth approximations of the original function can be considered and gradient type methods \cite{Nes:04} can be used for solving these smooth approximations. 
Some examples of smoothing techniques are Moreau envelope \cite{RocWet:98}, Forward-Backward envelope \cite{GisFal:18,SteThePat:17}, Douglas-Rachford envelope \cite{GisFal:18,PatSteBem:14}, Nesterov's smoothing \cite{Nes:05} and Gaussian smoothing \cite{NesSpo:17}. General properties for these smoothing techniques can be found in \cite{GisFal:18,Nes:05,NesSpo:17,PatSteBem:14,RocWet:98,SteThePat:17,NecFer:22,NecSuy:08}. Moreover, in e.g., \cite{Nes:05, NesSpo:17,NecFer:22,NecSuy:08} gradient type methods are considered for solving the smooth approximation. However, in large-scale problems the computation of the full gradient of the smooth approximation can be prohibitive.  In this case one can use coordinate descent algorithms \cite{FerQu:20,FerRic:15,NesSti:17,Nes:10} for solving the smooth approximation, as it was done e.g., in \cite{AbeBec:21,FerRic:17,NesSti:17}. On other hand, if the original function is still nonsmooth, but relative smooth w.r.t. a differentiable function, coordinate descent type methods still converge (sub)linearly \cite{LuFreNes:18}. Some examples of convex relative smooth functions were given in \cite{LuFreNes:18}. Relative coordinate descent algorithms were also considered in \cite{GaoLu:21,HanRic:21,HiePha:21}. However, in all these papers the original function is relative smooth along coordinates with respect to a separable function.

\medskip

\noindent \textit{Previous work:} In \cite{Nes:05} the following problem is considered
\vspace{-0.1cm}
\begin{equation}
F(x) = f(x) + \max_{u \in Q}\left\lbrace \langle Ax,u\rangle - {\phi}(u)\right\rbrace, \label{prob:NS}
\vspace{-0.1cm}
\end{equation}
with the gradient of $f$ assumed Lipschitz and a smooth approximation was proposed for this problem, which we call in this paper  Nesterov's smoothing. Moreover, \cite{Nes:05} also presents an accelerated gradient descent method for solving the smooth approximation for which sublinear rate  is obtained in the convex case. In the context of coordinate descent algorithms, \cite{NesSti:17} presents an accelerated coordinate descent algorithm  for solving the Nesterov's smooth  approximation.  Paper \cite{FerRic:17} considers the problem \eqref{prob:NS} with $f$ being separable, i.e., $f(x) =\sum_{i=1}^{N}f_{i}(x_{i})$, and considers the Nesterov's smoothing for the second term of the objective. Finally, a coordinate descent algorithm is used for solving the smooth approximation 
and rates of order $\mathcal{O}  \left( \frac{1}{\epsilon^2} \log \left( \frac{1}{\epsilon}\right)  \right)$ are derived in the convex case and  improved to $\mathcal{O}  \left( \frac{1}{\epsilon} \log \left( \frac{1}{\epsilon}\right)  \right)$, when the smooth approximation is strongly convex. Moreover, in \cite{AbeBec:21} an objective function of the form 
\vspace{-0.1cm}
\begin{align}
    \label{prob:FB}
    F(x) = \frac{1}{2}x^{T}Ax + b^{T}x  + \psi(x)
    \vspace{-0.1cm}
\end{align}
is considered, with $b \in \mathbb{R}^{n}$, $A\in\mathbb{R}^{n \times n}$ is a symmetric positive semidefinite matrix and $\psi$ is a proper closed convex function (possibly nonsmooth and nonseparable) which is proximal easy. Then, the Forward-Backward envelope is employed as a smooth approximation for the objective and a modified accelerated coordinate descent algorithm is used for solving this smooth approximation (in each iteration it is necessary to compute two proximal operators). Rate of order $\mathcal{O}  \left( \frac{1}{\epsilon^{3/2}} \right) $  is obtained in function values at some point $\Pi_{\text{dom} \psi}(x)$, where $\Pi_{\text{dom} \psi}(\cdot)$ is the projection onto the domain of $\psi$, provided that $\psi$ is Lipschitz in any bounded subset.  Finally, when the function is relative smooth along coordinates in \cite{HanRic:21} a relative randomized coordinate descent is proposed and sublinear rates are obtained in the convex case and linear rates for an objective that is relative strong with respect to a separable function. 

 \medskip
 
\noindent \textit{Contributions.} In this paper we first introduce a general  smoothing framework for the convex objective function $F$ in problem \eqref{eq:prob} and we use random (accelerated) coordinate descent methods from the literature to solve the smooth approximation. We also present a relative randomized coordinate descent algorithm for minimizing convex objective functions  that are relative smooth along coordinates w.r.t. a (possibly nonseparable) differentiable function. More precisely, our main contributions are: 

(i) We introduce a novel  smoothing framework for general  nonsmooth and nonseparable convex objective functions  (Assumption \ref{ass1}), which covers in particular the most important smooth approximations from the literature (Moreau envelope, Forward-Backward envelope, Douglas-Rachford envelope  and Nesterov's smoothing). 
\red{In the Moreau envelope, we assume that $F$ is a proper closed convex function and proximally easy. In the Forward-Backward and Douglas-Rachford envelopes our results are valid when $F$ has the form  \eqref{prob:FB}. In the Douglas-Rachford envelope, matrix  $\left( I_{n} + \gamma A\right)^{-1} $ must be computed easily, where $\gamma >0$ is the parameter of the smoothing and $I_{n}$ is the identity matrix. Finally, for Nesterov's smoothing our framework covers the case when $f$ has coordinate-wise Lipschitz gradient in the problem  \eqref{prob:NS}.} 
Then,  we consider random (accelerated)  coordinate descent algorithms for solving this smooth approximation. Under this general framework  we derive sublinear  convergence rates  in function values for the original problem \eqref{eq:prob}. 

(ii) Moreover, the rates are improved (to even linear) if the original objective function satisfies a $q$-growth condition (Assumption \ref{ass2}). We prove that Moreau envelope, Forward-Backward envelope and Nesterov's smoothing inherits such a $q$-growth. Table \ref{table:rates} summarizes  the main convergence results from this paper for each smoothing technique, where $\bar \kappa$ is the constant in  Assumption \ref{ass2}. 

(iii) We also consider nonsmooth objective functions (possibly nonseparable), but relative smooth along coordinates with respect to a given function  (possibly also nonseparable). For this problem a relative randomized coordinate descent algorithm is introduced, for which we derive sublinear rate when the objective $F$ is convex. Moreover, if the function $F$ satisfies a $q$-growth condition, we prove that the algorithm converges linearly with high probability.

\medskip 

\noindent The major difference between our result corresponding to the  Forward-Backward envelope and the result in \cite{AbeBec:21} is that a rate of order $\mathcal{O}(1/k^{1.5})$  is obtained in \cite{AbeBec:21} \red{in function values at the point $\Pi_{\text{dom} \psi}(x_{k})$, where $x_{k}$ is the iterate of the accelerated algorithm in \cite{AbeBec:21}, while we  derive an improved rate of order $\mathcal{O}(1/k^{2})$ in function values at the point $\text{prox}_{\gamma \psi}\left( x_{k} - \gamma \nabla f(x_{k})\right)$, where $x_{k}$ is the iterate of our Algorithm 2}. Another major difference is that the modified accelerated coordinate descent algorithm considered in \cite{AbeBec:21} requires computation of two proximal operators of $\psi$ per iteration, while in the algorithm considered in this paper only one proximal operator of $\psi$ is computed at each iteration. Moreover, in \cite{AbeBec:21} an extra assumption that $\psi$ is Lipschitz in any bounded subset is required, while we do not need this 
assumption  to obtain our convergence rates.

	\begin{table}[h!]
			\begin{tabular}{| c | c | c | c | c | c |}
				\hline
				\diagbox{Property}{Smoothing} & ME & FB & DR & NS & Result\\
				\hline
				\multicolumn{6}{|c|}{ Coordinate Descent Method} \\
				\hline
				Convex & $1/\epsilon$ & $1/ \epsilon$  & $1/\epsilon$  & $1 / \epsilon^2$ & Thm \ref{theo:2}  \\ \hline
				\multirow{1}{*}{$q$-growth}  & \multirow{1}{*}{$\ln \left(1/\epsilon\right) $} &\multirow{1}{*}{$\ln \left(1/\epsilon\right) $} & \multirow{1}{*}{-} &  \multirow{1}{*}{$(1/\epsilon)\ln \left(1/\epsilon\right)$} & Thm \ref{theo:qdif2F}\\
				\hline 
				\multicolumn{6}{|c|}{ Accelerated Coordinate Descent Method} \\ \hline
				Convex  & $1/\sqrt{\epsilon}$ & $1/ \sqrt{\epsilon}$  & $1/\sqrt{\epsilon}$  & $1 / \epsilon$ & Thm \ref{theo:4}  \\ \hline 
				\multicolumn{6}{|c|}{ Restart Accelerated Coordinate Descent Method knowing a lower bound $\bar\kappa$} \\ \hline
				$q$-growth & $\ln \left(1/\epsilon\right) $ & $\ln \left(1/\epsilon\right) $ & - &  $(1/\sqrt{\epsilon})\ln \left(1/\epsilon\right)$ & Cor \ref{cor:rest}  \\ \hline
				\multicolumn{6}{|c|}{Restart Accelerated Coordinate Descent Method not knowing $\bar\kappa$} \\ \hline
				$q$-growth  & \multicolumn{2}{|c|}{$\ln \left(\frac{1}{\epsilon}\right) \log_{2} \left( \ln \left( \frac{1}{ \epsilon }  \right)     \right) $}  & - &  $\left(\frac{1}{\sqrt{\epsilon}}\right)\ln \left(\frac{1}{\epsilon}\right) \log_{2} \left( \ln \left( \frac{1}{ \epsilon }  \right)     \right) $  & Thm \ref{theo:rest} \\\hline
			\end{tabular}
		\caption{Convergence rates derived in this paper for  Moreau envelope (ME), Forward-Backward (FB) envelope, Douglas-Rachford (DR) envelope  and Nesterov's smoothing (NS) in the convex and $q$-growth cases. }
		\label{table:rates}
	\end{table}

\medskip 

\noindent \textit{Content.} The paper is organized as follows. In Section \ref{sec:pre} we present some definitions and preliminary results. We introduce our general smoothing framework  in Section \ref{sec:SmoothApprox}. In Section \ref{sec:qgrowth} we analyze the $q$-growth condition. Then, in Section \ref{sec:CDmethods}, we consider random (accelerated) coordinate descent algorithms and   derive  convergence rates in function values of the original objective in the convex and $q$-growth cases. In Section \ref{sec:RelSmooth}, we present a relative smooth coordinate descent algorithm and derive convergence in the convex and $q$-growth cases. Finally, in Section \ref{sec:Simulations} we provide detailed numerical simulations.

\medskip
\subsection{Preliminaries}
\label{sec:pre}
In this section we present some notations, definitions and  some preliminary results.
Let $U \in \mathbb{R}^{n\times n}$ be  a column permutation of the $n \times n$ identity matrix and further let $U = [U_{1},...,U_{N}]$ be a decomposition of $U$ into $N$ submatrices, with $U_{i} \in \mathbb{R}^{n \times n_{i}}$, where $\sum_{i = 0}^{N}n_{i} = n$. Any vector $x \in \rset^n$ can be written uniquely as $x = \sum_{i = 0}^{N} U_{i} x^{(i)}$, where $x^{(i)} = U_{i}^{T}x \in \mathbb{R}^{n_{i}}$.  For spaces $R^{n_i}$, let us fix some norms $ \|\cdot\|_{(i)}$, $i = 1,\cdots,N$. Moreover, $\|x\|^{2} =  \sum_{i=1}^{N} \|x^{(i)}\|^{2}$ is defined as the Euclidean norm and $\|\cdot\|_{\ell_p}$ is the $\ell_p$ norm, \red{for $p \neq 2$}. 
We consider $\delta_{B}$ the indicator function of the set $B$ and $e_{i}$ is the $i$-th component vector of the canonical basis. \red{We denote with $X^{*}$ the set of minimizers of problem \eqref{eq:prob}  and  with $\bar{x} = [x]_{X^{*}}$, the projection of $x$ onto $X^{*}$ (all the letters with bar above denote the projection of a point onto $X^{*}$).} 
For a given function $F$ and $y \in \text{dom} F$ its level set is defined as
\begin{equation*}
\mathcal{L}_{F}(y) = \{x \in \text{dom} F: \; F(x)\leq F(y)\}.
\end{equation*}

\noindent Next, we present the definition of convexity along coordinates.

\begin{definition}
\label{def:concor}
For any fixed $x \in \mathbb{R}^{n}$ and $i =1,\cdots,N$ denote $\phi^{x}_{i}:\mathbb{R}^{n_{i}}\to \mathbb{R}$ as:
\vspace{-0.2cm}
\begin{equation}
\phi^{x}_{i}(d) = \psi(x + U_{i}d). \label{eq:phi}
\end{equation}
 We say that $\psi:\mathbb{R}^{n}\to \mathbb{R}$ is (strictly) convex along coordinates if the partial functions $\phi^{x}_{i}:\mathbb{R}^{n_{i}}\to \mathbb{R}$ are (strictly) convex for all $x \in \mathbb{R}^{n}$ and $i=1,\cdots,N$.
\end{definition}

\noindent Let us define the notion of block coordinate-wise Lipschitz gradient \cite{Nes:10}. 
\begin{definition}
    \label{def:lip} 
    Given a closed convex set $X \subseteq \mathbb{R}^{n}$ and
    a differentiable function $f:\mathbb{R}^{n} \to \mathbb{R}$. The gradient of $f$ is block coordinate-wise Lipschitz continuous on $X$, with constants $L_{i}>0$, if for all $x,x + U_{i} h \in X$ and $i=1,\cdots,N$, the following inequality holds
    \vspace{-0.2cm}
    \begin{align}
        \|U_{i}^{T}(\nabla f(x + U_{i} h) - \nabla f(x)) \|_{(i)} \leq L_{i} \|h\|_{(i)}. \label{eq:54}
    \end{align}
\end{definition}

\noindent  If \eqref{eq:54} holds, then we have the following relation \cite{Nes:10}:
\vspace{-0.2cm}
\begin{align}
\label{lip2}
|f(x+U_{i} h) - f(x) - \langle U_{i}^{T} \nabla f(x), h \rangle| \leq \frac{L_{i}}{2} \|h\|^2_{(i)}.
\end{align}

\noindent When $N=1$ in the inequality above, we say the function $f$ has Lipschitz gradient. Next result was proved for the case $N=1$, see for example \red{  \cite[Thm 2.1.5]{Nes:04}}. We adapt this result for the (block) coordinate case. 
\begin{lemma} 
	\label{lem:1}
	Consider $f:\mathbb{R}^{n} \to \mathbb{R}$ a function and $X \subseteq \mathbb{R}^{n}$ a closed convex set. Assume that $f$ is convex along coordinates and the following inequality holds for all $x, x+U_{i} h\in X$ and $i=1,\cdots,N$
    \vspace{-0.2cm}
	\begin{equation}
	\label{lip:3}	
	f(x+U_{i} h) - f(x) - \langle U_{i}^{T} \nabla f(x), h \rangle \leq \dfrac{L_{i}}{2} \|h\|_{(i)}^2. 
	\end{equation}

	\noindent Then, we have:
    \vspace{-0.2cm}
    \begin{equation*}
	f(x+U_{i}h) - \langle U_{i}^{T}\nabla f(x+U_{i}h),h \rangle +  \dfrac{1}{2L_{i}}  \|  U_{i}^{T}(\nabla f(x+U_{i}h) - \nabla f(x))\|^{2}_{(i)} \leq f(x),
    \end{equation*}
    \noindent and
    \vspace{-0.2cm}
	\begin{equation*}
	\|U_{i}^{T}(\nabla f(x + U_{i} h) - \nabla f(x)) \|_{(i)} \leq L_{i} \|h\|_{(i)}.
	\end{equation*}
\end{lemma}

\red{
\begin{proof}
	See appendix for a proof. 
\end{proof}
}

\section{Smooth approximations for convex problems}
\label{sec:SmoothApprox}
In this section we introduce a general smoothing framework for convex functions. Our framework covers the most important smoothing techniques from literature: Moreau, Forward-Backward, Douglas-Rachford envelope and Nesterov's smoothing.

\begin{assumption}
	\label{ass1}
	Consider $F_\gamma(x)$, a smooth approximation for the function $F$ (with parameter $\gamma>0$). Assume that:
	\begin{itemize}
		\item [] A.1 Consider $F_\gamma^{*}=\min F_\gamma(x) $. We have that $F_\gamma^{*} \leq F^{*}$.  
		\item [] A.2 There exist operators $\mathcal{B}$ and $\mathcal{C}$ and  some constant $D \geq 0$ such that  \vspace{-0.2cm}
        \begin{align*}
		F(\mathcal{B}(x)) - \gamma D \leq F_{\gamma}(x) \leq F(\mathcal{C}(x)), \quad \forall x \in \text{dom} F.
        \vspace{-0.2cm}
        \end{align*}
		\item [] A.3 The gradient of $F_\gamma$ is block coordinate-wise Lipschitz continuous on $\mathbb{R}^{n}$, with constants $L_i(F_\gamma)>0$, see Definition \ref{def:lip}.
		\item [] A.4 $F_\gamma(x)$ is convex. 
	\end{itemize}
\end{assumption}

\noindent Next, we  give some examples of smooth approximations that satisfy Assumption \ref{ass1}. 
The reader may find others examples of smoothing techniques that fit our framework. 

\subsection{Moreau envelope}
\label{sec:Moreau}
As a first example of a smoothing technique we have the Moreau Envelope \cite{RocWet:98}. The Moreau envelope of a function $F$ with $\gamma>0$ is defined as 
\begin{equation}
F_{\gamma}(x) := \min_{u \in \mathbb{R}^{n}} \left\lbrace F(u) + \dfrac{1}{2\gamma}\|u-x\|^2\right\rbrace, \label{eq:ME}
\end{equation}

\noindent while the minimizer of the problem above, called proximal mapping, is 
\begin{equation}
\text{prox}_{\gamma F}(x) := \arg \min_{u \in \mathbb{R}^{n}} \left\lbrace F(u) + \dfrac{1}{2\gamma}\|u-x\|^2\right\rbrace. 
\end{equation}

\noindent Below, we show that the Moreau Envelope satisfies Assumption \ref{ass1}:  

   \noindent A.1  It is well-known that \cite{RocWet:98}:
	\begin{equation}
    \label{eq:30}
	F^{*} = \min_{x \in \mathbb{R}^{n}} F(x) = \min_{x \in \mathbb{R}^{n}} F_{\gamma}(x) = F^{*}_{\gamma} \quad \text{and} \quad \arg \min_{x \in \mathbb{R}^{n}} F(x) = \arg \min_{x \in \mathbb{R}^{n}} F_{\gamma}(x). 
	\end{equation}
	\noindent 
	
	 \noindent A.2 Consider $\hat{u} = \text{prox}_{\gamma F}(x)$. A direct implication from the definition of $F_{\gamma}$ is that $F(\hat{u}) \leq F_{\gamma}(x) \leq F(x)$, for all $\gamma > 0$. Hence $\mathcal{B}(x) =  \text{prox}_{\gamma F}(x)$, $\mathcal{C}$ is the identity operator and
    $D=0$.
	 
	 \noindent A.3 If $F$ is a proper closed convex function, then $F_{\gamma}$ is differentiable and its gradient \cite{RocWet:98}
	\begin{equation}
	\nabla F_{\gamma}(x) = \gamma^{-1}\left(x - \text{prox}_{\gamma F}(x)\right)   
	\end{equation}
	
	\noindent is $\gamma^{-1}$-Lipschitz continuous. Moreover, note that the coordinate-wise Lipschitz constant of $\nabla F_{\gamma}$ is also $L_{i}(F_{\gamma}) =\gamma^{-1}$. In fact, 
	\begin{align*}
	&\|U_{i}^{T}(\nabla F_\gamma(x + U_{i} h) - \nabla F_\gamma(x)) \|_{(i)} \leq	\|\nabla F_\gamma(x + U_{i} h) - \nabla F_\gamma(x) \| \leq \gamma^{-1} \|h\|_{(i)}.
	\end{align*}

	 \noindent A.4 If $F$ is convex, then $F_{\gamma}$ is also convex (see for example \cite{RocWet:98}).

\subsection{Forward-Backward envelope}
\label{sec:Forward-Backward}
Consider the problem \eqref{eq:prob} defined as the following composite problem
\begin{equation}
F^* = \min_{x \in \mathbb{R}^n} F(x) := f(x) + \psi(x),  \label{eq:CP} 
\end{equation}

\noindent where $\psi$ is a proper closed convex function. Moreover, $\psi$ is proximal easy and $f$ has Lipschitz gradient with constant $L>0$. The Forward-Backward envelope  of the pair $(f,\psi)$ with a smooth parameter  $\gamma > 0$ is the function \cite{GisFal:18,SteThePat:17}:
\begin{equation}
F_{\gamma}(x) = f(x) - \dfrac{\gamma}{2}\|\nabla f(x) \|^2 + \psi^{\gamma}(x - \gamma \nabla f(x)), \label{eq:FBE1}
\end{equation}
where $\psi^{\gamma}$ is the Moreau envelope of the function $\psi$ defined in \eqref{eq:ME}. Note that the Forward-Backward envelope can also be defined as
\vspace{-0.2cm}
\begin{equation}
F_{\gamma}(x)  = \min_{u \in \mathbb{R}^{n}} f(x) + \langle \nabla f(x), u-x \rangle + \dfrac{1}{2\gamma} \|u-x\|^2 + \psi(u).	\label{eq:FBE2}
\vspace{-0.2cm}
\end{equation}

\noindent Now, we summarize the results obtained in \cite{GisFal:18,SteThePat:17} that fits in our framework.
If $f$ is twice differentiable and has Lipschitz gradient, and $\psi$ is a proper closed convex function, then $F_{\gamma}$ is continuously differentiable with
\begin{equation}
\nabla F_{\gamma}(x) = \gamma^{-1}\left(I- \gamma \nabla^2 f(x)\right)\left[ x -  \text{prox}_{\gamma \psi}(x - \gamma \nabla f(x))\right] \label{eq:grad}.
\end{equation}

\noindent \red{Moreover, Forward-Backward envelope satisfies Assumption \ref{ass1}  when $f$ is quadratic}:

	\noindent A.1 For all $\gamma L \in \left(0, 1\right]$, we have
	\begin{equation}
    \label{eq:31}
	\arg \min_{x \in \mathbb{R}^{n}} F(x) \subset \arg \min_{x \in \mathbb{R}^{n}} F_{\gamma}(x) \quad
	\text{and} \quad F^{*} = \min_{x \in \mathbb{R}^{n}} F(x) = \min_{x \in \mathbb{R}^{n}} F_{\gamma}(x) = F^{*}_{\gamma} . 
	\end{equation}
	\noindent Moreover, if $\gamma L \in \left(0,1\right)$, we have
	\begin{equation}
    \label{eq:32}
	\arg \min_{x \in \mathbb{R}^{n}} F(x) = \arg \min_{x \in \mathbb{R}^{n}} F_{\gamma}(x).
	\end{equation}
	
	\noindent A.2 Consider $\hat{u} =  \text{prox}_{\gamma \psi}\left( x - \gamma \nabla f(x)\right) $, recall that $F(\hat{u}) \leq F_{\gamma}(x)$ for $\gamma L \in \left(0, 1\right]$. Moreover, from the definition of $F_{\gamma}$, we have $F_{\gamma}(x) \leq F(x)$, for all $\gamma > 0$. Hence, $\mathcal{B}(x) = \text{prox}_{\gamma \psi}\left( x - \gamma \nabla f(x)\right)$, $\mathcal{C}$ is the identity operator and $D=0$.
	
	\noindent A.3 Consider 
	\begin{equation}
	F(x) = f(x) + \psi(x) = \frac{1}{2}x^{T}Ax + b^{T}x + \psi(x), \label{eq:1}
	\end{equation}
	
	\noindent for $A \in \mathbb{R}^{n \times n}$ a symmetric positive semidefinite matrix and $b \in \mathbb{R}^{n}$ and $\psi$ is a proper closed convex function. Then, \eqref{eq:grad} becomes
	\begin{equation}
	\nabla F_{\gamma}(x) = \gamma^{-1}\left(I- \gamma A\right)\left[ x -  \text{prox}_{\gamma \psi}(x - \gamma \nabla f(x))\right]. \label{eq:gradquad} 
	\end{equation}
	
	\noindent If $\gamma L \in \left(0,1\right)$,  then the gradient \eqref{eq:gradquad} is Lipschitz continuous with constant $\gamma^{-1}(1-\gamma \lambda_{\min}(A))$, where $\lambda_{\min}(A)$ is the smallest eigenvalue of A.  Moreover, for  $\gamma L \in \left(0,1\right)$, using Proposition 4.3 in \cite{GisFal:18} and Lemma \ref{lem:1}, we get that the $i$th coordinate-wise Lipschitz constant of $\nabla F_{\gamma}$ in  \eqref{eq:gradquad} is  $$L_{i}(F_{\gamma})  = \dfrac{\| U_{i}^{T} \left( I_{n} - \gamma A\right)U_{i} \|  }{\gamma} = \dfrac{\|I_{n_{i} } - \gamma A_{ii}\|  }{\gamma}.$$

    \noindent\red{Note that the gradient, $\nabla F_{\gamma}$, from \eqref{eq:gradquad} is in fact  Lipschitz and block coordinate-wise Lipschitz for all $\gamma>0$. However, the Lipschitz and coordinate-wise Lipschitz constants take larger values  when $\gamma L \geq 1$ than when $\gamma L < 1$. In fact, a more general result holds, i.e.,   the gradient, $\nabla F_{\gamma}$, given in \eqref{eq:grad} is Lipschitz (coordinate-wise Lipschitz) on a set $X\subset \mathbb{R}^{n}$ when $f$ is convex, the hessian of $f$ is Lipschitz (coordinate-wise Lipschitz) and there exists a constant $R>0$ such that $\|x -  \text{prox}_{\gamma \psi}(x - \gamma \nabla f(x))\| \leq R \quad \forall x \in X$. However, in our analysis we consider $\gamma L \in \left(0,1\right)$ and $f$ quadratic, since under these conditions convexity is also preserved, see the next statement.}

\medskip 
 
	\noindent A.4 When $\gamma L \in \left(0, 1\right)$ and 
	$F$ is defined as in  \eqref{eq:1}, we have $F_{\gamma}$ also convex, \red{see   \cite[Proposition  4.4]{GisFal:18}.}

\subsection{Douglas-Rachford envelope}
\label{sec:Douglas-Rachford}
Let us present the results obtained in \cite{GisFal:18,PatSteBem:14} that are relevant in this paper. Consider the problem \eqref{eq:CP} where $\psi$ is a proper closed convex function and $f$ has Lipschtz gradient with constant $L >
0$. The Douglas-Rachford envelope for the problem \eqref{eq:CP} is defined as:
\begin{equation}
F_{\gamma}(x) = f^{\gamma}(x) - \gamma\|\nabla f^{\gamma}(x) \|^2 + \psi^{\gamma}(x - 2\gamma \nabla f^{\gamma}(x)),
\end{equation}

\noindent where $f^{\gamma}, \psi^{\gamma}$ are the Moreau Envelopes of the functions $f$ and $\psi$, defined in \eqref{eq:ME}. The Douglas-Rachford envelope can also be expressed as
\vspace{-0.2cm}
\begin{equation*}
F_{\gamma}(x) =  \min_{u \in \mathbb{R}^{n}} f\left( P_{\gamma} (x) \right)	+ \langle \nabla f \left( P_{\gamma} (x) \right), u -  P_{\gamma} (x) \rangle + \dfrac{1}{2\gamma} \|u -  P_{\gamma} (x) \|^2 + \psi(u),
\vspace{-0.2cm}
\end{equation*}

\noindent with $P_{\gamma} (x) \!=\!  \text{prox}_{\gamma f} (x)$.  Moreover, Douglas-Rachford envelope satisfies Assumption \ref{ass1} when f is quadratic: \\

 \noindent A.1 If $\gamma L \in \left(0, 1\right)$   we have: 
    \vspace{-0.2cm}
	\begin{equation*}
	 F^{*} = \min_{x \in \mathbb{R}^{n}} F(x) = \min_{x \in \mathbb{R}^{n}} F_{\gamma}(x) = F_{\gamma}^{*} \quad \text{and} \quad
	\arg \min_{x \in \mathbb{R}^{n}} F(x) = \text{prox}_{\gamma f}\left( \arg \min_{x \in \mathbb{R}^{n}} F_{\gamma}(x)\right).
     \vspace{-0.2cm}
	\end{equation*}	
    A.2 Consider $\hat{u} =  \text{prox}_{\gamma f}(x)$ and $\hat{v} =  \text{prox}_{\gamma \psi}\left( 2\hat{u} - x \right) $, if $\gamma L \in \left(0, 1\right)$ we have \red{$F(\hat{v}) \leq F_{\gamma}(x)$}.
    Moreover, $F_{\gamma}(x) \leq F(\hat{u})$, for all $\gamma>0$. 
    Hence $\mathcal{B}(x) = \text{prox}_{\gamma \psi}\left( 2\text{prox}_{\gamma f}(x) - x \right)$, $\mathcal{C}(x) = \text{prox}_{\gamma f}(x)$ and $D=0$. 
 
\noindent A.3 Consider
	$F$ defined as in \eqref{eq:1}. If $\gamma L \in \left(0, 1\right)$ we have $\nabla F_{\gamma}$ is Lipschitz continuous:
    \vspace{-0.2cm}
	\begin{equation*}
	\nabla F_{\gamma}  (x)= \dfrac{1}{\gamma}\left( \dfrac{2}{\gamma} \left( \gamma^{-1} I_{n} +  A\right)^{-1}  - I_{n}\right) G(x),
    \vspace{-0.2cm}
	\end{equation*}
	
	\noindent with
    \vspace{-0.2cm}
	\begin{align*}
	G(x) &= \left( \gamma^{-1} I_{n} +  A\right)^{-1} (\gamma^{-1}x- b) - \text{prox}_{\gamma \psi}\left( 2\left( \gamma^{-1}I_{n} +  A\right)^{-1} (\gamma^{-1} x- b) -x\right).
	\end{align*}
	\noindent Matrix $\left( I_{n} + \gamma A\right)^{-1} $ can be computed easily e.g.,  when $A$ is block diagonal.  For  \eqref{eq:1}, using Proposition 4.5 in \cite{GisFal:18} and Lemma \ref{lem:1}, we get that the coordinate-wise Lipschitz constant of $\nabla F_{\gamma}$ is $L_{i}(F_{\gamma}) =\dfrac{\|U_{i}^{T}\left( P+P^2 \right) U_{i}\|}{\gamma}$, with $P=
	2\left( I_{n} + \gamma A\right)^{-1} - I_{n}$. \\
	
	\noindent A.4 When $\gamma L\in \left(0, 1\right)$ and 
	$F$ is defined as \eqref{eq:1}, we have $F_{\gamma}$ also convex, \red{see   \cite[Proposition  4.6]{GisFal:18}.}

\subsection{Nesterov's smoothing}
Consider the function
\vspace{-0.2cm}
\begin{equation}
F(x) = f(x) + \max_{u \in Q}\left\lbrace \langle Ax,u\rangle - {\phi}(u)\right\rbrace, \label{eq:NSprob}
\vspace{-0.2cm}
\end{equation}

\noindent  where $Q \subset \mathbb{R}^{m}$ is a closed convex bounded subset,  ${\phi}$ a continuous convex function on Q, and $f$ a convex function having a coordinate-wise Lipschitz gradient, with constants $L_{i}(f)>0$, for $i=1,\cdots,N$. Examples of function that satisfies \eqref{eq:NSprob} can be found in \cite{Nes:05}. Let us recall the main properties of  Nesterov's smoothing  \cite{Nes:05,NesSti:17}:
\vspace{-0.2cm}
\begin{equation}
F_{\gamma}(x) = \red{f(x)} + \max_{u \in Q}\left\lbrace \langle Ax,u\rangle - {\phi}(u) - \gamma d(u)\right\rbrace, \label{eq:smNes}
\vspace{-0.2cm}
\end{equation}

\noindent where $d$ is strongly convex function on $Q$ w.r.t. some norm \red{$\|\cdot\|_{W}$ }with parameter $\sigma>0$ and $d(u) \geq 0$ for all $u\in Q$. Then, Nesterov's smoothing satisfies Assumption \ref{ass1}:\\

\noindent A.1 Defining  $ \bar D = \max_{u} d(u)$, we have that: 
\vspace{-0.2cm}
\begin{equation}
	F_{\gamma}(x) \leq F(x) \leq F_{\gamma}(x) + \gamma \bar D, \quad \forall x \in \mathbb{R}^{n}, \label{ineq2:NS}
 \vspace{-0.2cm}
\end{equation}
\noindent Hence $F_{\gamma}^{*} \leq F^{*}$.

\noindent A.2  From \eqref{ineq2:NS}, we have $D=\bar D$. Moreover, $\mathcal{B}$ and $\mathcal{C}$ are the identity operators. 
	
\noindent A.3  The gradient of $F_{\gamma}$ is
    \vspace{-0.2cm}
	\begin{equation}
	\nabla F_{\gamma}(x) = \nabla f(x) + A^{T} u_{\gamma}(x), \label{grad:NS}
    \vspace{-0.2cm}
	\end{equation}
	
	\noindent where $u_{\gamma}(x)$ is the unique solution of the optimization problem in \eqref{eq:smNes}. Moreover, the gradient of $F_{\gamma}$ is coordinate-wise Lipschitz with constant 
    \vspace{-0.2cm}
    \begin{align}
    \label{eq:49}
    L_{i}(F_{\gamma}) =  L_{i}(f) +   \dfrac{1}{\gamma}\left( \|Ae_{i}\|^{*}_{U}\right)^2, \;   \text{where} \;  \|Ae_{i}\|_{U}^{*} = \max_{u} \{ \langle Ae_{i},u\rangle : \|u\|_{U} \leq 1\}. 
    \vspace{-0.2cm}
    \end{align}	
\noindent A.4 \red{Note that $F_{\gamma}$ is a maximum of linear functions in $x$, hence it is convex. }
\\

\noindent Based on the previous examples, we can see that our framework is general and covers a wide range of smoothing techniques.


\section{Functional $q$-growth}
\label{sec:qgrowth}
Strong convexity type conditions, such as $q$-growth property,  allow to derive  linear rates  for first order methods, see \cite{NecNesGli:19}. In this section we prove that smooth approximations of objective functions satisfying some $q$-growth condition inherits such a property.  \red{Recall that $X^{*}$ is the set of minimizers of problem \eqref{eq:prob}  and  $\bar{x} = [x]_{X^{*}}$ is the projection of $x$ onto $X^{*}$.}
In this section we consider the following additional assumption on  $F$. 

\begin{assumption}
	\label{ass:qgrowth}  
 Consider  $X \subseteq \text{dom} F$. We assume that the function $F$ satisfies a \textit{functional $q$-growth} on $X$, i.e.,  there exists a constant $\kappa>0$ and $q\in [1,2]$ such that: 
    \vspace{-0.2cm}
	\begin{equation*}
	F(x) - F^{*} \geq \dfrac{\kappa}{q}\| x - \bar{x}\|^{q} \quad \forall x \in X.
    \vspace{-0.2cm}
	\end{equation*}
\end{assumption}

\noindent  Below, we prove that Moreau envelope, Forward-Backward envelope and Nesterov's smoothing satisfy also a $q$-growth like condition. \red{Let us introduce some notations that  will be used in the sequel. We define $\displaystyle L_{\max} = \max_{i=1:N} L_{i}(F_{\gamma})$.  We consider the following norms, for $\alpha\in \mathbb{R}$:
\vspace{-0.2cm}
\begin{equation*}
\|x\|^{2}_{\alpha} =  \sum_{i=1}^{N} (L_{i}(F_{\gamma}))^{\alpha} \|x^{(i)}\|^{2}_{(i)}.
\end{equation*}}

\subsection{Moreau and Forward-Backward envelopes} 

\begin{theorem}
	\label{theo:ME&FB}
	Let Assumption \ref{ass:qgrowth} hold for a given function $F$ and $X=\mathcal{L}_{F}(x_{0})$. For $q\in [1,2)$ assume additionally  that the level set of the smooth approximation $F_{\gamma}$ satisfies:  
    \vspace{-0.2cm}
	\begin{equation*}
	\|x-\bar{x}\| \leq R \quad \forall x \in \mathcal{L}_{F_{\gamma}}(x_{0}) \quad \red{\text{and some} \quad R>0.}
    \vspace{-0.2cm}
	\end{equation*}
	If $F$ is a proper closed convex function and $F_{\gamma}$ is the Moreau envelope, then it satisfies the following inequality: 
    \vspace{-0.2cm}
	\begin{equation}
    \label{eq:35}
	F_{\gamma}(x) - F^{*}_{\gamma} \geq \dfrac{\hat{\kappa}}{2}  \| x - \bar{x}\|^{2}_{1} \quad \forall x \in \mathcal{L}_{F_{\gamma}}(x_{0}), 
    \vspace{-0.2cm}
	\end{equation}
	with 
    \vspace{-0.2cm}
	\begin{equation}
    \label{eq:45}
	\hat \kappa =  \left\lbrace\begin{array}{ll} \dfrac{\kappa\gamma }{(\gamma \kappa +1) } \; 
	\text{ if } q=2 \\
	 \dfrac{\kappa^2\gamma^2}{(\kappa\gamma + R^{2-q})^2} \; \text{ if }  q \in [1,2).
	\end{array}\right. 
	\end{equation}

    \noindent If $F = f + \psi$, such that $f$ has Lipschitz gradient with constant $L>0$,  $\psi$ is a proper closed convex function and $F_{\gamma}$ is the Forward-Backward envelope with $\gamma L \in (0, 1)$, then it satisfies  \eqref{eq:35} with 
    \vspace{-0.2cm}
	\begin{equation}
    \label{eq:46}
	\hat \kappa =  \left\lbrace\begin{array}{ll} \dfrac{\kappa (1-\gamma L)}{(\gamma \kappa + 1 - \gamma L)L_{\max}(F_{\gamma}) } \; 
	\text{ if } q=2 \\
	\dfrac{\kappa^2\gamma (1-\gamma L)  }{(\kappa\gamma + R^{2-q} (1-\gamma L) )^2 L_{\max}(F_{\gamma}) } \; \text{ if }  q \in [1,2).
	\end{array}\right.   			
	\end{equation}
	
\end{theorem}

\begin{proof}
Let us first prove the case of  Moreau envelope. Consider $x\in \mathcal{L}_{F_{\gamma}}(x_{0})$ and $z=\text{prox}_{\gamma F}(x)$. Note that $z\in \mathcal{L}_{F}(x_{0})$. Since $F$ is a proper closed convex function and satisfies the $q$-growth condition, then we have: 
    \vspace{-0.2cm}
	\begin{align*}
	F_{\gamma}(x) - F^{*}_{\gamma} =  F(z) + \dfrac{1}{2\gamma}\|z-x\|^2- F^{*}
    &\geq  \dfrac{\kappa}{q}\| z - \bar{z}\|^{q}  + \dfrac{1}{2\gamma}\|z-x\|^2 \\
	&\geq  \min_{u \in \mathbb{R}^{n}} \dfrac{\kappa}{q}\| u - \red{\bar{z}}\|^{q}  + \dfrac{1}{2\gamma}\|u-x\|^2.
    \vspace{-0.2cm}
	\end{align*}
		
	\noindent For $q=2$, from the optimality condition (considering $u$ the optimal point of the previous problem), we get:
        \vspace{-0.2cm}
		\begin{align*}
		\kappa(u-\red{\bar{z}}) + \dfrac{1}{\gamma}(u-x) = 0
		\iff
		u = \dfrac{1}{\kappa\gamma+1} x + \dfrac{\gamma\kappa}{\gamma\kappa +1}\red{\bar{z}},
        \vspace{-0.2cm}
		\end{align*}
		
		\noindent and we obtain 
        \vspace{-0.2cm}
		\begin{align*}
		\min_{u \in \mathbb{R}^{n}} \dfrac{\kappa}{2}\| u - \red{\bar{z}}\|^{2}  + \dfrac{1}{2\gamma}\|u-x\|^2 = \dfrac{\kappa}{2(\gamma \kappa +1)} \|x - \red{\bar{z}}\|^2  \geq   \dfrac{\kappa }{2(\gamma \kappa +1) L_{\max}(F_{\gamma})} \|x - \bar{x}\|^2_1.
        \vspace{-0.2cm}
		\end{align*}
		
		\noindent Since $L_{i}(F_{\gamma}) = \gamma^{-1}$, for $i=1,\cdots,N$, in the Moreau envelope, then $F_{\gamma}$ is 2-growth w.r.t. the norm $\|\cdot\|_1$, with constant $\hat\kappa = \dfrac{\kappa\gamma}{(\gamma \kappa +1) }$. 
		
		 \noindent  For $q\in [1,2)$,  considering $u = \red{\bar{z}}$, we get
        \vspace{-0.2cm}
		\begin{align*}
		  \min_{u \in \mathbb{R}^{n}} \dfrac{\kappa}{q}\| u - \red{\bar{z}}\|^{q}  + \dfrac{1}{2\gamma}\|u-x\|^2 = \dfrac{1}{2\gamma}\|\red{\bar{z}}-x\|^2  \geq  \dfrac{1}{2\gamma}\|\bar{x}-x\|^2 = \dfrac{1}{2}\|\bar{x}-x\|^2_{1}.
        \vspace{-0.2cm}
		\end{align*}
		
		\noindent On other hand,  considering $u \neq \red{\bar{z}}$, from the optimality conditions there exists $u$:
		\vspace{-0.2cm}
		\begin{equation*}
		\kappa \|u-\red{\bar{z}}\|^{q-2} \left(u - \red{\bar{z}} \right) + \dfrac{1}{\gamma}\left( u-x\right) =0.
        \vspace{-0.2cm}
		\end{equation*}
		
		\noindent Define $\alpha = \kappa \|u-\red{\bar{z}}\|^{q-2}$ and $\beta = \dfrac{1}{\gamma}$, then from equality above we get
        \vspace{-0.2cm}
		\begin{align*}
		\alpha (u - \red{\bar{z}} ) + \beta ( u-x ) =0 \Leftrightarrow (\alpha +\beta) u = \alpha\red{\bar{z}} + \beta x 
		\;\; \Leftrightarrow \;\;  u = \dfrac{\alpha}{\alpha +\beta} \red{\bar{z}} + \dfrac{\beta}{\alpha + \beta} x.
        \vspace{-0.2cm}
		\end{align*}
		
		\noindent This implies that 
        \vspace{-0.2cm}
		\begin{align}
		&\min_{u \in \mathbb{R}^{n}} \dfrac{\kappa}{q}\| u - \red{\bar{z}}\|^{q}  + \dfrac{1}{2\gamma}\|u-x\|^2 \nonumber \\ &=  \dfrac{\kappa}{q}\left\| \dfrac{\alpha}{\alpha +\beta} \red{\bar{z}} + \dfrac{\beta}{\alpha + \beta} x - \red{\bar{z}}\right\|^{q}  + \dfrac{1}{2\gamma}\left\|\dfrac{\alpha}{\alpha +\beta} \red{\bar{z}} + \dfrac{\beta}{\alpha + \beta} x-x\right\|^2 \nonumber \\
		&=   \dfrac{\kappa \beta^{q}}{q(\alpha+\beta)^{q}}\left\|  x - \red{\bar{z}}\right\|^{q}  + \dfrac{\alpha^2}{2\gamma(\alpha + \beta)^2}\left\| \red{\bar{z}} -x\right\|^2. \label{eq:9} \vspace{-0.2cm}	 
		\end{align}
		
		\noindent Since $u$ is the minimizer of the problem above, we get 
        \vspace{-0.2cm}
		\begin{equation*}
		\dfrac{\kappa}{q}\| u - \red{\bar{z}}\|^{q} \leq \dfrac{\kappa}{q}\| u - \red{\bar{z}}\|^{q}  + \dfrac{1}{2\gamma}\|u-x\|^2 \leq \dfrac{\kappa}{q}\| x - \red{\bar{z}}\|^{q} \Rightarrow \| u - \red{\bar{z}}\| \leq \|x - \red{\bar{z}}\|.
        \vspace{-0.2cm}
		\end{equation*}
		
		\noindent From the boundedness assumption of the level set we have that there exists  $R>0$ such that $\|y-\red{\bar{z}}\| \leq R$, for all $y \in L_{F_{\gamma}}(x_{0})$. Using the fact that $x \in L_{F_{\gamma}}(x_{0})$ and the inequality above, we obtain  $\|u - \red{\bar{z}}\| \leq R$. Moreover,
        \vspace{-0.2cm}
		\begin{equation*}
		\dfrac{\alpha}{\alpha + \beta } = \dfrac{\kappa}{\|u-\red{\bar{z}}\|^{2-q}} \cdot \dfrac{\gamma \|u - \red{\bar{z}}\|^{2-q}}{\gamma\kappa + \|u-\red{\bar{z}}\|^{2-q}} = \dfrac{\kappa\gamma}{\gamma\kappa + \|u-\red{\bar{z}}\|^{2-q}} \geq \dfrac{\kappa\gamma}{\kappa\gamma + R^{2-q}}. \vspace{-0.2cm}
		\end{equation*}
		
		\noindent Hence, we obtain
		\vspace{-0.2cm}
        \begin{align}
		\min_{u \in \mathbb{R}^{n}} \dfrac{\kappa}{q}\| u - \red{\bar{z}}\|^{q}  + \dfrac{1}{2\gamma}\|u-x\|^2 &\geq \dfrac{\kappa^2\gamma^2}{2\gamma(\kappa\gamma + R^{2-q})^2}\| \red{\bar{z}} -x\|^2 \\
		& \geq \dfrac{\kappa^2\gamma}{2(\kappa\gamma + R^{2-q})^2L_{\max}(F_{\gamma})} \| \bar{x} -x\|^2_1. \nonumber \vspace{-0.2cm}
		\end{align}
		Since $L_{i}(F_{\gamma}) = \gamma^{-1}$ in the Moreau Envelope, then \red{$F_{\gamma}$ satisfies 2-growth condition} w.r.t.  norm $\|\cdot\|_1$ with the constant $\hat \kappa =  \dfrac{\kappa^2\gamma^2}{(\kappa\gamma + R^{2-q})^2}$.

	\noindent Now, let us prove for Forward-Backward envelope. Consider $x\in \mathcal{L}_{F_{\gamma}}(x_{0})$ and $z=\text{prox}_{\gamma \psi}\left( x - \gamma \nabla f(x)\right)$. Note that, if $\gamma L \in (0,1)$, then $z\in \mathcal{L}_{F}(x_{0})$. Since $F$ satisfies functional $q$-growth,  $\psi$ is a proper closed function and $f$ has Lipschitz gradient, for $\gamma L \in \left( 0, 1\right)$, we have
    \vspace{-0.2cm}
	\begin{align*}
	F_{\gamma}(x) - &  F_{\gamma}(\bar{x}) =  f(x) + \langle \nabla f(x), z-x \rangle + \dfrac{1}{2\gamma} \|z-x\|^2 + \psi(z) - F^{*} \\
	&\geq \left( \dfrac{1}{2\gamma}- \dfrac{L}{2}\right)  \|z-x\|^2 + F(z) - F^{*} \geq  \dfrac{1}{2} \left( \dfrac{1}{\gamma}- L\right)  \|z-x\|^2 + \dfrac{\kappa}{q}\| z - \bar{z}\|^{q}\\
	&\geq \min_{u \in \mathbb{R}^{n}}  \dfrac{1}{2} \left( \dfrac{1}{\gamma}- L\right)  \|u-x\|^2 + \dfrac{\kappa}{q}\| u - \red{\bar{z}}\|^{q}. \vspace{-0.2cm}
	\end{align*}
	where in the first equality we use \eqref{eq:31}, \eqref{eq:32} and the definition of $F_{\gamma}$. Following a similar analysis as in the Moreau envelope case we can get the results. 
\end{proof}

\noindent Note that, in the Forward-Backward case we do not need to assume the function $f$ to be convex. 
\noindent Since in this paper the function $F$ is convex, we have that Assumption \ref{ass:qgrowth} is equivalent to Kurdyka-Lojasiewicz (KL) inequality with $p=\frac{q-1}{q}$ and $\mu > 0$ \cite{BolNgu:16}:
\vspace{-0.2cm}
\begin{equation}
    \label{KL}
    \mu(F(x) - F^{*})^{p} \leq \|\nabla F(x)\| \quad \forall x \in X.
    \vspace{-0.2cm}
\end{equation}

\noindent In \red{\cite[Theorem 3.4]{LiPon:18}}, if the original function $F$ satisfies KL with $p\in (0,1/2]$, then the Moreau Envelope  satisfies KL with $p=1/2$, i.e., \red{Moreau envelope satisfies $2$-growth condition}. We have obtained  similar results in Theorem \ref{theo:ME&FB}. Further, considering $F=f + \psi$, with $f$ a twice continuously differentiable function with Lipschitz gradient and  $\psi$ proper closed convex function, \red{\cite[Theorem 3.2]{LiPon:17}} and \red{\cite[Theorem 5.2]{YuLiPon:19}}  prove that if $F$ satisfies KL with $p\in (0,1/2]$,  then
the Forward-Backward is a KL function with exponent $1/2$, i.e., \red{Forward-Backward envelope satisfies $2$-growth property}, which confirms our results from Theorem \ref{theo:ME&FB}. However, in Theorem \ref{theo:ME&FB} we give an explicit expression for the constant $\kappa$ and we provide a proof that does not need the assumption of twice differentiability for $f$ in the Forward-Backward case.

\subsection{Nesterov's smoothing}
For the next theorem, for some $\gamma>0$, consider the following level set of a function $F$:
\vspace{-0.2cm}
\begin{equation}
\label{eq:41}
\mathcal{L}_{F}(x_{0},\gamma) = \{x:F(x)\leq F(x_{0}) + \gamma \bar{D}\}.
\vspace{-0.2cm}
\end{equation}
and the following (nonconvex) set
\vspace{-0.2cm}
\begin{align}
    \hat{X} = \mathcal{L}_{F_{\gamma}}(x_{0}) \cap \{ x \in \mathbb{R}^{n}: F_{\gamma}(x) - F^{*} \geq \gamma \bar{D}   \}. \label{eq:42}
    \vspace{-0.2cm}
\end{align}

\begin{theorem}
	\label{theo:NS}
	Let assumption \ref{ass:qgrowth} hold for a given function $F$ and  $X=\mathcal{L}_{F}(x_{0},\gamma)$. Considering $F$ defined as in \eqref{eq:NSprob} and $F_{\gamma}$ the Nesterov's smoothing.  Then, the following inequality holds
    \vspace{-0.2cm}
	\begin{equation*}
    F_{\gamma}(x) - F^{*}_{\gamma} \geq \hat{\kappa}\|x - \bar{x}\|^{q}_1 - \gamma \bar{D} \quad \forall x \in \mathcal{L}_{F_{\gamma}}(x_{0})\quad \text{with} \quad \hat{\kappa} = \dfrac{\kappa}{q L_{\max}^{\frac{q}{2}}}. \vspace{-0.2cm}
	\end{equation*}

    \noindent Moreover, $F_{\gamma}$ satisfies
    \vspace{-0.2cm}
	\begin{equation}
    \label{eq:47}
	 F_{\gamma}(x) - F^{*}_{\gamma} \geq \hat{\kappa}\|x - \bar{x}\|^{q}_1 \quad \forall x \in \hat{X}, \quad \text{with} \quad \hat{\kappa} = \dfrac{\kappa}{2 q L_{\max}^{\frac{q}{2}}}. \vspace{-0.2cm}
	\end{equation}
    
    \noindent Additionally, assume that for  the level set of $F_{\gamma}$,  $\mathcal{L}_{F_{\gamma}}(x_{0})$, we have
    \vspace{-0.2cm}
	\begin{equation}
	\|x-\bar{x}\| \leq R \quad \forall x \in \mathcal{L}_{F_{\gamma}}(x_{0}),
    \label{eq:38} \vspace{-0.2cm}
	\end{equation}
	then  $F_{\gamma}$ also satisfies
    \vspace{-0.2cm}
	\begin{equation}
    \label{eq:48}
	 F_{\gamma}(x) - F^{*}_{\gamma} \geq \dfrac{\hat{\kappa}}{2}\|x - \bar{x}\|^{2}_{1} \quad \forall x \in \hat{X}, \quad \text{with} \quad \hat{\kappa} = \dfrac{\kappa}{q L_{\max} R^{2-q}}. \vspace{-0.2cm}
	\end{equation}
\end{theorem}

\begin{proof} 
	Assume the function $F$ satisfy the $q$-growth condition. Since 
    \vspace{-0.2cm}
	\begin{equation*}
	F(x) \leq F_{\gamma}(x) + \gamma \bar D,
	\vspace{-0.2cm}
    \end{equation*} 
	then
    \vspace{-0.2cm}
	\begin{align}
    F_{\gamma}(x) - F^{*} + \gamma \bar{D}  
	&\geq \dfrac{\kappa }{q}\|x - \bar{x}\|^{q} \geq \dfrac{\kappa L^{\frac{q}{2}}_{\max}}{qL^{\frac{q}{2}}_{\max}} \left(\sum_{i=1}^{N} \|x^{(i)} - \bar{x}^{(i)}\|^2 \right)^{\frac{q}{2}} \nonumber \\
	&\geq \dfrac{\kappa}{qL^{\frac{q}{2}}_{\max}} \left(\sum_{i=1}^{N} L_{\max} \|x^{(i)} - \bar{x}^{(i)}\|^2 \right)^{\frac{q}{2}} \geq \dfrac{\kappa}{qL^{\frac{q}{2}}_{\max}} \|x- \bar{x}\|^q_{1}. \label{eq:39} \vspace{-0.2cm}
	\end{align}
	Since $F^{*}_{\gamma} \leq F^{*}$, the first statement follows. From the last inequality, for $x\in \hat{X}$:  
    \vspace{-0.2cm}
    \begin{align}
    \label{eq:43}
    2\left(F_{\gamma}(x) - F^{*}\right) \geq     F_{\gamma}(x) - F^{*} + \gamma \bar{D} \geq \dfrac{\kappa}{qL^{\frac{q}{2}}_{\max}} \|x- \bar{x}\|^q_{1}. \vspace{-0.2cm} 
\end{align}

\noindent Using the fact that $F^{*}_{\gamma} \leq F^{*}$, the second statement also follows. Moreover, using the additional assumption in \eqref{eq:38} and the first inequality in \eqref{eq:39}, for $x\in \hat{X}$, we have 
\vspace{-0.2cm}
\begin{align}
    \label{eq:44}
  2\left(F_{\gamma}(x) - F^{*}\right) \geq \dfrac{\kappa }{q}\|x - \bar{x}\|^{q} \geq    \dfrac{\kappa }{q R^{2-q}}\|x - \bar{x}\|^{2} \geq \dfrac{\kappa }{q R^{2-q} L_{\max} }\|x - \bar{x}\|^{2}_{1}. \vspace{-0.2cm}
\end{align}

\noindent From $F^{*}_{\gamma} \leq F^{*}$, the  statement follows 
\end{proof}

\noindent An immediate consequence of  Theorems \ref{theo:ME&FB} and \ref{theo:NS} is the following corollary.

\begin{corollary}
\label{cor:1}
    Given a function $F$ and $\bar{X} \subseteq \mathbb{R}^{n}$. Consider the following inequality for a smooth approximation of $F$, $F_{\gamma}$, and some $\bar{\kappa} > 0$:
    \vspace{-0.2cm}
    \begin{equation}	
    \label{eq:40}
    F_{\gamma}(x) - F^{*} \geq \dfrac{\bar{\kappa}}{2^{\frac{\bar{q}}{2}}} \| x - \bar{x}\|^{\bar{q}}_1  \quad \forall x \in \bar{X} \quad \text{and} \quad \bar{x} = [x]_{X^{*}}. \vspace{-0.2cm}
	\end{equation}

   \noindent Then, we have: \\
    (i) If  $F_{\gamma}$ is the Moreau or Forward-Backward envelopes satisfying the assumptions in Theorem \ref{theo:ME&FB}, then $F_{\gamma}$ satisfies \eqref{eq:40} with $\bar{q}=2$, $\bar{X} = \mathcal{L}_{F_{\gamma}}(x_0)$. Moreover $\bar\kappa=\hat{\kappa}$, with $\hat{\kappa}$ defined in \eqref{eq:45} and \eqref{eq:46} 
    for Moreau and Forward-Backward envelopes, respectively. \\
    (ii) Let Assumption \ref{ass:qgrowth} hold with $X=\mathcal{L}_{F}(x_{0},\gamma)$, with $\mathcal{L}_{F}(x_{0},\gamma)$ defined in \eqref{eq:41}. Moreover, consider $F$ and $\hat{X}$ defined as in \eqref{eq:NSprob} and \eqref{eq:42}, respectively.  If $F_{\gamma}$ is the Nesterov's smoothing, then 
    $F_{\gamma}$ satisfies \eqref{eq:40} with $\bar{q}=q$, $\bar{X} = \hat{X}$ and $\bar \kappa = 2^{\frac{\bar{q}}{2}} \hat{\kappa}$, where $\hat{\kappa}$ is defined in \eqref{eq:47}. Additionally if $F_{\gamma}$ satisfies \eqref{eq:38}, then $F_{\gamma}$ satisfies \eqref{eq:40} with $\bar{q}=2$,  $\bar{X} = \hat{X}$ and $\bar\kappa=\hat{\kappa}$, where $\hat{\kappa}$ is defined in \eqref{eq:48}.
\end{corollary}
\begin{proof}
(i) In the Moreau and Forward-Backward envelopes,
from \eqref{eq:30}, \eqref{eq:31} and \eqref{eq:32}, we have $F_{\gamma}(\bar{x}) = F_{\gamma}^{*} = F^{*} = F(\bar{x})$. Hence, the statements follows from Theorem \ref{theo:ME&FB}.\\
(ii) The statement follows from \eqref{eq:43} and \eqref{eq:44}.
\end{proof}

\noindent It is still an open question if similar results can be derived for  the Douglas-Rachford envelope. Note that for the Nesterov's smoothing we have $L_{i}(F_{\gamma}) = \mathcal{O}\left(\gamma^{-1}\right)$, see \eqref{eq:49}, hence $\bar{\kappa} = \mathcal{O}\left( \gamma^{\frac{\bar{q}}{2}} \right)$.  Based on the previous corollary,
we impose the following assumption for a smooth approximation.

\medskip 

\begin{assumption}
	\label{ass2}
	Given a function $F$ and  $\bar{X} \subseteq  \mathbb{R}^{n}$. Assume that the smooth approximation of $F$, $F_{\gamma}$, satisfies the following inequality for some $\bar{\kappa}> 0$ and $\bar{q}\in [1,2]$  	
    \vspace{-0.2cm}
    \begin{equation*}	F_{\gamma}(x) - F^{*} \geq \dfrac{\bar{\kappa}}{2^{\frac{\bar q}{2}}} \| x - \bar{x}\|^{\bar q}_1  \quad \forall x \in \bar{X} \quad \text{and} \quad \bar{x} = [x]_{X^{*}}.
    \vspace{-0.2cm}
	\end{equation*}
\end{assumption}


\section{(Accelerated) coordinate descent methods} 
\label{sec:CDmethods}
In this section we analyze the convergence of (accelerated) coordinate descent methods for minimizing the smooth approximation and consequently the original objective function in the convex and $q$-growth cases.


\subsection{Coordinate descent: convex case}
In the framework of smoothing as given in Assumption \ref{ass1}, the smooth approximation $F_{\gamma}$ has coordinate-wise gradient. Therefore,  we can use (accelerated) coordinate descent algorithms \cite{FerQu:20,Nes:10,NesSti:17} for solving the smooth approximation of problem \eqref{eq:prob}:
\vspace{-0.2cm}
\begin{align}
\label{eq:prob2}
F_{\gamma}^* = \min_{x \in \mathbb{R}^n} F_{\gamma}(x). \vspace{-0.2cm}
\end{align}
In the next algorithm, we \red{consider a} random counter $\mathcal{R}_{\alpha}$, with $\alpha \in \mathbb{R}$, which generates an integer number $i \in \{1, . . . , N\}$ with probability \cite{Nes:10}:
\vspace{-0.2cm}
\begin{equation}
p^{(i)}_{\alpha} = L_{i}^{\alpha} (F_{\gamma}) \cdot \left[ \sum_{i=1}^{N} L_{j}^{\alpha}(F_{\gamma})\right]^{-1}. \label{eq:probability} \vspace{-0.2cm}
\end{equation}
Thus, \red{ $\mathcal{R}_{\alpha}$ is a discrete
random variable over $\{1, ..., N\}$ and its distribution is specified by probabilities
as in \eqref{eq:probability}. Note that $\mathcal{R}_{0}$ generates a uniform distribution}. Let us recall the algorithm proposed in \cite{Nes:10}.

\begin{center}
	\noindent\fbox{%
		\parbox{7.5 cm}{%
			\textbf{Algorithm 1}:\\
			Given a starting point $x_{0} \in \text{dom} F_{\gamma}$.
			For $k\geq 0$ do:
			\begin{enumerate}
				\item \red{Sample $i_k$ from $\mathcal{R}_{\alpha}$.}
				\item  Compute $h^{(i_{k})} (x_{k})= -U_{i_{k}}^{T}\nabla F_{\gamma}(x_{k})$. 
               \item Update: \\
				$x_{k+1} = x_{k} + \dfrac{1}{L_{i_{k}}(F_{\gamma})} U_{i_{k}} h^{(i_{k})} (x_{k})$.
			\end{enumerate}
	}}
\end{center}
\medskip
\noindent From inequality (2.4) in \cite{Nes:10}, we have for all $k \geq 0$ the descent:
\vspace{-0.2cm}
\begin{align}
    \label{eq:dec}
    F_{\gamma}(x_{k+1}) \leq F_{\gamma}(x_{k}). \vspace{-0.2cm}
\end{align}

\noindent Denote the set of optimal solutions of \eqref{eq:prob2} by $\bar X^{*}_{\gamma}$ and let  $x_{*}$ be an element of this set. Define also: 
\vspace{-0.2cm}
\begin{equation*}
R_{\alpha} = \max_{k \geq 0} \min_{x_{*} \in \bar{X}^{*}_{\gamma}} \|x_{k} - x_{*}\|_{\alpha}< \infty
\text{ and }
S_{\alpha}(\gamma) = \sum_{i=1}^{N}L_{i}^{\alpha} (F_{\gamma}). \vspace{-0.2cm} 
\end{equation*}

\begin{theorem}
	\label{theo:1}
	Let Assumption \ref{ass1} hold. Then, the iterates of  Algorithm 1 satisfy
    \vspace{-0.2cm}
	\begin{equation*}
	\mathbb{E}[F_{\gamma}(x_{k})] - F_{\gamma}^{*} \leq \dfrac{2 S_{\alpha}(\gamma) R_{1-\alpha}^2 }{k+4}. \vspace{-0.2cm}
	\end{equation*} 
	
\end{theorem}

\begin{proof}
	The result follows using Theorem 1 in \cite{Nes:10} and Assumption \ref{ass1} [A3-A4]. 
\end{proof}

\noindent Next, we can also provide the convergence rate for the original function in  problem \eqref{eq:prob}. 

\begin{theorem}
	\label{theo:2}
	Let Assumption \ref{ass1} hold. Then, the iterates of Algorithm 1 satisfy
    \vspace{-0.2cm}
	\begin{equation}
	\mathbb{E}[F(\mathcal{B}(x_{k}))] - F^{*} \leq \dfrac{2 S_{\alpha}(\gamma) R_{1-\alpha}^2 }{k+4}+ \gamma D. \label{eq:ConvRate} \vspace{-0.2cm}
	\end{equation} 
\end{theorem}

\begin{proof}
	By Assumption \ref{ass1}[A.1-A.2] we have $F^{*} \geq F_{\gamma}^{*}$ and $\mathbb{E}[F(\mathcal{B}(x_{k})] \leq \mathbb{E}[F_{\gamma}(x_{k})] + \gamma D$. Using  Theorem \ref{theo:1},  the result  follows. 
\end{proof}

\noindent If we want to get $\epsilon$ accuracy, since in the Nesterov's smoothing $D>0$,  we need to take $\gamma = \mathcal{O}\left( \dfrac{\epsilon}{2D}\right)$. Additionally, we have (see \eqref{eq:49})
\vspace{-0.2cm}
\begin{align*}
L_{i}(F_{\gamma}) = \mathcal{O}\left(\gamma^{-1}\right), \quad \text{then}  \quad
S_{\alpha}(\gamma) = \mathcal{O}\left(\gamma^{-\alpha}\right), \vspace{-0.2cm}
\end{align*}

\noindent and 
\vspace{-0.2cm}
\begin{align*}
R_{1-\alpha}^2 =  \|x_{\bar{k}} - x_{*}\|_{1-\alpha}^2 = \sum_{i=1}^{N}L_{i}^{1-\alpha}(F_{\gamma})  \|x_{\bar{k}}^{(i)} - x_{*}^{(i)}\|_{(i)}^{2}  = \sum_{i=1}^{N}\mathcal{O}\left(\gamma^{\alpha-1}\right) \|x_{\bar{k}}^{(i)} - x_{*}^{(i)}\|_{(i)}^{2}, \vspace{-0.2cm}
\end{align*}
for some $\bar{k}>0$ and for some  $x_{*} \in \bar{X}^{*}_{\gamma}$. Therefore,
\vspace{-0.2cm}
\begin{equation*}
S_{\alpha}(\gamma) R_{1-\alpha}^2 = \sum_{i=1}^{N}\mathcal{O}\left(\gamma^{-1}\right) \|x_{\bar{k}}^{(i)} - x_{*}^{(i)}\|_{(i)}^{2}. \vspace{-0.2cm}
\end{equation*}
Hence, inequality \eqref{eq:ConvRate} becomes (recall that in this case $\mathcal{B}$ is the identity operator):
\vspace{-0.2cm}
\begin{equation}
\mathbb{E}[F(x_{k})] - F^{*} \leq \mathcal{O}\left(\dfrac{1}{\gamma(k+4)}\right) + \gamma D. \vspace{-0.2cm}
\end{equation} 
Since $\gamma = \mathcal{O}\left( \dfrac{\epsilon}{2D}\right)$, it implies that the complexity  is of order $\mathcal{O}\left( \epsilon^{-2} \right)$. On other hand for the Moreau, Forward-Backward and Douglas-Rachford envelopes we have $F(\mathcal{B}(x_{k})) \leq F_{\gamma}(x_{k})$, i.e., $D=0$.  Hence in these cases, we obtain 
\vspace{-0.2cm}
\begin{equation*}
\mathbb{E}[F(\mathcal{B}(x_{k}))] - F^{*} \leq \dfrac{2 S_{\alpha}(\gamma) R_{1-\alpha}^2 }{k+4}. \vspace{-0.2cm}
\end{equation*} 

\noindent Then, in the Moreau envelope, we have $S_{\alpha}(\gamma) = \dfrac{N}{\gamma^{\alpha}}$ for some $\gamma>0$ (see Section \ref{sec:Moreau}); in the Forward-Backward,  $S_{\alpha}(\gamma) =\displaystyle\sum_{i=1}^{N}\dfrac{\|I_{n_{i} \times n_{i}} - \gamma A_{ii}\|^{\alpha}  }{\gamma^{\alpha}}$  for some $\gamma\in \left(0, 1/L\right)$ (see Section \ref{sec:Forward-Backward}); and in the Douglas-Rachford we have $S_{\alpha}(\gamma) = \dfrac{\|U_{i}^{T}\left( P+P^2 \right) U_{i}\|^{\alpha}}{(2\gamma)^{\alpha}}$ for some $\gamma\in \left(0, 1/L\right)$ and $P=
2\left( I_{n} + \gamma A\right)^{-1} - I_{n}$ (see Section \ref{sec:Douglas-Rachford}). 
Since the constant $\gamma$ does not  depend on $\epsilon$,  the complexity in these cases is of order $\mathcal{O}\left( \epsilon^{-1}\right)$.    

\medskip

\noindent Below we discuss the computation cost of \red{$U_{i_{k}}^{T}\nabla F_{\gamma}(x_{k})$} at each iteration for the previous four examples. 
 
\noindent \textit{Example 1:} Consider the Moreau envelope. In this case we have
    \vspace{-0.2cm}
	\red{\begin{equation*}
	U_{i}^{T}\nabla F_{\gamma}(x) = \gamma^{-1}\left(x^{(i)} - U_{i}^{T}\text{prox}_{\gamma F}(x)\right).  \vspace{-0.2cm}
	\end{equation*}}
	\noindent Hence, this smoothing is efficient in the context of coordinate descent when \red{a block of components} of the $\text{prox}$ of  $F$ can be evaluated easily based on the previously computed $\text{prox}$,  given that only a block of coordinates is modified in the prior iteration. For example:
	\begin{enumerate}
		\item $F(x) = \sum_{i=1}^{N} F_{i}(U_{i}^{T}x ) $,   with  $\text{prox}$  of $F_{i}$ computed in $\mathcal{O}(n_{i})$ operations. 
		\item $F(x) = g(Ax+ b)$, with $b\in\mathbb{R}^{m}$ and $A\in\mathbb{R}^{m\times n}$, such that 
		$AA^{T} = 1/\alpha I$. Then:
		\vspace{-0.1cm}
        \begin{equation*}
		\text{prox}_{F}(x) =   x - \alpha A^{T} \left( Ax + b  - \text{prox}_{\alpha^{-1}g}(Ax+b)\right). \vspace{-0.1cm}
		\end{equation*}
		As explained in \cite{NesSti:17} \red{the computational cost of updating $Ax$ in each iteration is given by $\mathcal{O}(m n_{i})$ operations}. If prox of $g$ is computed in $\mathcal{O}(n)$ operations, then 
		the complexity of updating \red{$U_{i}^{T}\nabla F_{\gamma}(x)$ is given by  $\mathcal{O}( m n_{i} + n )$} operations. One example of such  $g$  whose prox can be computed in $\mathcal{O}(n)$ operations is $g(x) = \|x\|_{\ell_1}$. 
	\end{enumerate}
	\textit{Example 2}: Consider the Forward-Backward envelope, for problem \eqref{eq:1}, we have 
    \vspace{-0.2cm}
	\red{\begin{equation*}
    \label{eq:55}
	U_{i}^{T}	\nabla F_{\gamma}(x) = \gamma^{-1}U_{i}^{T}\left(I- \gamma A\right)\left[ x -  \text{prox}_{\gamma \psi}(x - \gamma (Ax+b) ))\right]. \vspace{-0.2cm}
	\end{equation*}}
	\noindent\red{The computational cost of  $Ax$ in each iteration is of order $\mathcal{O}(n n_{i})$ = $\mathcal{O}(n)$ }. In this case, it is interesting to apply  Algorithm 1 for solving the Forward-Backward envelope when the prox of $\psi$ can also be computed in $\mathcal{O}(n)$ operations, such as: 
	\begin{enumerate}
		\item $\psi(x) = \lambda \|x\|$. Then, the prox of $\psi$ is given by  
        \vspace{-0.2cm}
		\begin{align*}
		\text{prox}_{\lambda \|x\|}(x) = \left( 1 - \dfrac{\lambda}{\|x\|}\right)_{+} x.  \vspace{-0.2cm}
		\end{align*}
        \item $\psi(x) = \sum_{s\in \mathcal{S}} \| x_{s}\|$, where $\mathcal{S}$ is a partion of $\{1,\cdots,n\}$. This function serves as regularizer to induce group sparsity. For $s\in \mathcal{S}$, the components of the proximal mapping indexed by $s$ are
        \vspace{-0.2cm}
        \begin{align*}
		(\text{prox}_{\lambda \psi}(x))_s = \left( 1 - \dfrac{\lambda}{\|x_{s}\|}\right)_{+} x_{s}.  
        \vspace{-0.2cm}
		\end{align*}
        
		\item $\psi  = \delta_{B_{2}[0,r]}  = \delta_{\{x: \|x\| \leq r \}} $ for some $r>0$. The projection onto $B_2$ is given by:
        \vspace{-0.2cm}
		\begin{equation*}
		\text{prox}_{\delta_{B_{2}[0,r]}}(x)= \left\lbrace\begin{array}{ll} x \; 
		\text{ if } \; \|x\| \leq r  \\ 
		\dfrac{r x}{\|x\|} \; \text{ otherwise.}  
		\end{array}\right.   \vspace{-0.2cm}
		\end{equation*}
		\item $\psi = \delta_{B_{1}[c,r]}  = \delta_{\{x: \|x-c\|_{\ell_1} \leq r \}}$ for some $c>0$. Then, the prox can be computed in $\mathcal{O}(n)$ operations, see  \cite{DucSha:08}.
		\item $\psi = \delta_{C}$, where $C=\{x \in \mathbb{R}^{n} : a^{T}x = b, l \leq x \leq u  \} $. Then, the prox can be computed in $\mathcal{O}(n)$ operations.
		\item $\psi = \delta_{C}$ where $C = \{ x: Ax=b \}$ with a full row rank $A \in \mathbb{R}^{m \times n}$ and $b \in \mathbb{R}^{m}$. Assuming that $m \ll n$, then for any $\gamma>0$, $\text{prox}_{\gamma\psi} (x) = P_{C}(x) = \dfrac{1}{\gamma}\left( x - A^{T}(AA^{T})^{-1}(Ax-b)\right) $.
		After a preprocessing step in which $AA^{T}$ is computed, we can compute $P_{C}(x)$ in $\mathcal{O}(n)$ operations.
		\item $\psi(x) = \lambda TV(x)$, where $\lambda>0$, the total variation regularization. In this case one can also compute the prox in  $\mathcal{O}(n)$ operations, see \cite{BarSra:14,Joh:13,KolPocRol:16}.  
		\item $\psi(x) = \|x\|^{r+2}$, where $r \geq 0$. Then, computing the prox is equivalent to solving a polynomial of degree $r+1$, i.e., $\text{prox}_{\gamma\psi} (x) = -\theta c$,  where $\theta$ is a root of the polynomial $1-\theta - \|c\|^r \theta^{r+1} =0$ (when $c\neq 0$) . See also \cite{LuFreNes:18} for other examples. 
	\end{enumerate}
	
	\noindent \textit{Example 3}: Consider the Douglas-Rachford envelope, for problem \eqref{eq:1}, we have
    \vspace{-0.2cm}
	\red{\begin{align}
    \label{eq:56}
	U_{i}^{T}\nabla F_{\gamma}  (x) = U_{i}^{T}\left( 2H  - I_{n}\right) \gamma H \left((x-\gamma b) - \text{prox}_{\gamma \psi}\left( 2H (x-\gamma b) -x\right) \right), \vspace{-0.3cm}
	\end{align}}
	\noindent with $H=\left( I_{n} + \gamma A\right)^{-1}$. In this case, we assume that $H$ can be computed easily, e.g., the matrix $A$ is diagonal. Note also that in this case the computation of the vector $Hx$ can be done in \red{$\mathcal{O}(n n_{i})=\mathcal{O}(n)$} operations and consequently  Algorithm 1 can be efficiently implemented for minimizing the Douglas-Rachford envelope when the prox of $\psi$ can be also computed in $\mathcal{O}(n)$ operations, see examples for the Forward-Backward envelope. \\

	\noindent \textit{Example 4}: Consider the Nesterov's smoothing. As explained in \cite{NesSti:17}, if the set $Q$ and function $\phi$ in \eqref{eq:smNes} are simple and the product $Ax$ is known (in the coordinate descent methods this vector can be updated at each iteration in \red{$\mathcal{O}(m n_{i})$} operations), then the vector $u_{\gamma}(x)$ is computed in $\mathcal{O}(m)$ operations.  Moreover, one component of the second term of the gradient in \eqref{grad:NS}, \red{$U_{i}^{T}A^{T} u_{\gamma}(x)$, is computed in $\mathcal{O}(m n_{i})$ operations}. 
	The following examples satisfy this structure:
	\begin{enumerate}
		\item Consider $F(x) = \|\tilde{A}x - \tilde{b}\|_{\infty} = \max_{u\in Q}\{\langle Ax,u \rangle - \langle b, u \rangle \} $, where
        \vspace{-0.2cm}
		\begin{equation*}
		\tilde{A} \in \mathbb{R}^{m\times n}, \,\ \tilde{b} \in \mathbb{R}^{m}, \,\ A = \begin{bmatrix}
		\tilde{A} \\ -\tilde{A}
		\end{bmatrix} \in \mathbb{R}^{2m \times n}, \,\ b = \begin{bmatrix}
		\tilde{b} \\ - \tilde{b}
		\end{bmatrix} \in \mathbb{R}^{2m} \vspace{-0.2cm}
		\end{equation*}
		
		\noindent and $Q := \{u_{j} \in \mathbb{R}^{2m} : \sum_{j} u_{j} =1, u_{j} \geq 0 \}$ is the unit simplex in $\mathbb{R}^{2m}$. 
		The smooth approximation of $F$ is given by 
        \vspace{-0.2cm}
		\begin{equation*}
		F_{\gamma}(x) = \gamma \log \left( \dfrac{1}{2m} \sum_{j=1}^{2m} \exp\left( \dfrac{e_{j}^{T}Ax - b_{j} }{\gamma} \right) \right). \vspace{-0.2cm}
		\end{equation*}

        \item Consider $F(x) = \|Ax-b\|_{\ell_1} = \max_{u \in Q} \{\langle Ax,u \rangle - \langle b, u \rangle \}$, where $Q = [-1,1]^{n}$. The smooth approximation of $F$ is given by 
        \vspace{-0.2cm}
		\begin{equation*}
		F_{\gamma} (x) = \sum_{j=1}^{m} \|e_{j}^{T}A\| \cdot \phi_{\gamma}\left( \dfrac{|e_{j}^{T}Ax - b_{j}|}{\|e_{j}^{T}A\|}\right), \quad \phi_{\gamma}(t) = \left\lbrace\begin{array}{ll} \dfrac{t^{2}}{2\gamma}, \;  0\leq t \leq \gamma \\
		t - \dfrac{\gamma}{2}, \; \gamma \leq t . 
		\end{array}\right. \vspace{-0.2cm}
		\end{equation*} 
		
		\item Consider the total variation problem $F(x) =   \dfrac{1}{2}\|Bx-c\|^2 + \lambda \|Dx\|_{\ell_1}$, with $B \in \mathbb{R}^{p \times n}$, $D \in \mathbb{R}^{m \times n}$ and $c\in \mathbb{R}^{p}$.  We can use the smoothing function described above for the nonsmooth term, $\|Dx\|_{\ell_1}$. 

        \item Consider the function $F(x) =  \|x\| = \max_{\|u\| \leq 1} \langle x,u \rangle$. Choosing $d(u) = \dfrac{1}{2}\|u\|^2$, the smooth approximation is given by:
        \vspace{-0.2cm}
        \begin{align*}
            F_{\gamma}(x) = \left\lbrace\begin{array}{ll}   \dfrac{1}{2 \gamma} \|x\|^2, \quad  \|x\| \leq \gamma \\
		      \|x\| - \dfrac{\gamma}{2} , \quad \gamma \leq \|x\| . 
		\end{array}\right. \vspace{-0.2cm}
        \end{align*}
	\end{enumerate}

\noindent Next, let us discuss  the computational cost of $\mathcal{B}(x_{k})$. Note that, in the Forward-Backward and Douglas-Rachford envelopes,
$\mathcal{B}(x_{k})$ is computed at each iteration since we need this vector for updating
\red{$U_{i}^{T}\nabla F_{\gamma}(x_{k})$}.  Moreover, in the Nesterov's smoothing $\mathcal{B}$ is the identity operator, which obviously has no additional computational cost. On other hand, in the Moreau envelope we have an additional cost since $\mathcal{B}$ is the full prox, i.e., $\mathcal{B}(x_{k}) =  \text{prox}_{\gamma F}(x_{k})$ and in \red{$U_{i}^{T}\nabla F_{\gamma}(x_{k})$ we only need to compute one block of components of the full prox}. However, it is not necessary to compute $\mathcal{B}(x_{k})$ at each iteration, we only need this computation when the algorithm stops.


\subsection{Accelerated coordinate descent: convex case}
In this section we consider the algorithm proposed in \cite{NesSti:17}.  In the next algorithm, $\sigma_{1-\alpha}\geq0$ is the \red{strong convexity parameter} of $F_{\gamma}$ with respect to the norm $\|\cdot\|_{1-\alpha}$.
\begin{center}
	\noindent\fbox{%
		\parbox{12 cm}{%
			\textbf{Algorithm 2} \\
	        Define $\nu_{0} = x_{0} \in \text{dom} F_{\gamma}$, $z_{0} = x_{0}$, $A_{0}=0$, $B_{0} = 1$ and $\beta = \dfrac{\alpha}{2}$.    \\
			For $k\geq 0$ do:
			\begin{enumerate}
				\item \red{Sample $i_k$ from $\mathcal{R}_{\alpha}$.}
				\item Find parameter $a_{k+1} > 0$ from equation $a^{2}_{k+1} S_{\beta}^2  = A_{k+1} B_{k+1}$, \\
				where $A_{k+1} = A_{k} + a_{k+1}$ and $B_{k+1} = B_{k} + \sigma_{1-\alpha} a_{k+1}$.
				\item Define $\alpha_{k} = \dfrac{a_{k+1}}{ A_{k+1}}$, $\beta_{k} = \dfrac{ \sigma_{1-\alpha} a_{k+1}}{B_{k+1}}$ and $y_{k} = \dfrac{(1-\alpha_{k})x_{k} + \alpha_{k}(1-\beta_{k})\nu_{k}}{1 - \alpha_{k}\beta_{k}}$.
				\item  Compute $h^{(i_{k})} (y_{k})= - U_{i_{k}}^{T}\nabla F_{\gamma} (y_{k})$. 
         \item Update  $x_{k+1} = y_{k} + \dfrac{1}{ L_{i_{k}}(F_{\gamma})  } U_{i_{k}} h^{(i_{k})} (y_{k})$ and 
				\\$\nu_{k+1} = (1-\beta_{k})\nu_{k} + \beta_{k}y_{k} + \dfrac{a_{k+1}}{(L_{i_{k}}(F_{\gamma}))^{1-\alpha} B_{k+1}\pi_{\beta}[i_{k}]}U_{i_{k}} h^{(i_{k})} (y_{k}). $ 		 	
				
			\end{enumerate}
	}}
\end{center}

\vspace{0.2cm}

\begin{theorem}
	\label{theo:3}
	Let Assumption \ref{ass1} hold and $x_{*}$ be a minimizer of problem \eqref{eq:prob2}. Then, the iterates of Algorithm 2 satisfy
    \vspace{-0.2cm}
	\begin{equation*}
	\mathbb{E}[F_{\gamma}(x_{k})] - F_{\gamma}^{*} \leq \dfrac{2 S_{\beta}^{2}(\gamma) \|x_{0} - x_{*}\|^{2}_{1-\alpha} }{k^{2}}. \vspace{-0.2cm}
	\end{equation*} 
	
\end{theorem}

\begin{proof}
	The result follows from \red{\cite[Theorem 2]{NesSti:17}} and Assumption \ref{ass1}[A.3-A.4]. 
\end{proof}

\noindent Next theorem presents the convergence rate for the original function.

\begin{theorem}
	\label{theo:4}
	Let Assumption \ref{ass1} hold and $x_{*}$ be a minimizer of  problem \eqref{eq:prob2}. Then, the iterates of Algorithm 2 satisfy
    \vspace{-0.2cm}
	\begin{equation*}
	\mathbb{E}[F(\mathcal{B}(x_{k}))] - F^{*} \leq \dfrac{2 S_{\beta}^{2}(\gamma) \|x_{0} - x_{*}\|^{2}_{1-\alpha} }{k^{2}} + \gamma D. \vspace{-0.2cm}
	\end{equation*} 
\end{theorem}

\begin{proof}
	By Assumption \ref{ass1}[A.1-A.2] we have $F^{*} \geq F^{*}_{\gamma}$ and $\mathbb{E}[F(\mathcal{B}(x_{k})] \leq \mathbb{E}[F_{\gamma}(x_{k})] + \gamma D$. Using the Theorem \ref{theo:3}, it follows the result. 
\end{proof}

\noindent Note that we need to take $\gamma = \mathcal{O}\left( \dfrac{\epsilon}{2D}\right)$, when $D>0$.  Making a similar analysis as in the non-accelerated case, we have that the complexity for Nesterov's smoothing is of order $\mathcal{O}\left( \epsilon^{-1}\right)$.
On other hand for the Moreau, Forward-Backward and Douglas-Rachford envelopes we get that the complexity is of the order $\mathcal{O}\left( \epsilon^{-\frac{1}{2}}\right)$,
since $D=0$ and $\gamma$ does not depend on $\epsilon$.

\subsection{Coordinate descent: $q$-growth case}
\noindent Under the additional $q$-growth  condition from Assumption \ref{ass2}  we derive below improved convergence rates for the previous algorithms. Recall that $X^{*}$ is the set of minimizers of the original problem \eqref{eq:prob},  $\bar{x} = [x]_{X^{*}}$ is the projection of $x$ onto $X^{*}$ and $\mathcal{C}$ is the identity operator for Moreau envelope, Forward-Backward envelope and Nesterov's smoothing. For simplicity, let us define
\begin{align}
&\eta = \min\left\{\bar \kappa,  \bar{\kappa}^{\frac{2}{\bar q}}  \right\}, \,\  R = \max_{k\geq0} \dfrac{1}{2}  \|x_{k} - \bar x_{k}\|^2_1, \,\ \Delta_{0} = F_{\gamma}(x_{0}) - F^{*}  +  R, 
\label{eq:33}\\
&C_{1} = \max \left\{ 1 - \dfrac{ \bar\kappa \Delta_{0}^{\frac{q-2}{2}}}{N(1+\bar \kappa)    }, 1 - \dfrac{\eta}{N(1+\bar \kappa)}      \right\}. \nonumber
\end{align}

\begin{theorem}
	\label{theo:qdif2}
	Let Assumptions \ref{ass1} and \ref{ass2} hold with $\mathcal{C}$ being the identity operator. Assuming also that $i_{k}$ is chosen uniformly at random, then the iterates of  Algorithm 1 satisfy
    \vspace{-0.2cm}
	\begin{align*}
	\dfrac{1}{2}\mathbb{E}\left[ \|x_{k} - \bar{x}_{k}\|^2_{1} 	\right] +  \mathbb{E}\left[ F_{\gamma}(x_{k}) -  F^{*}\right]    \leq C_1^{k} \left(  \dfrac{1}{2}\|x_{0}- \bar{x}_{0}\|^{2}_{1}   +   F_{\gamma}(x_{0}) - F^{*}\right). \vspace{-0.2cm}
	\end{align*} 
\end{theorem}

\begin{proof}
	We have that 
    \vspace{-0.2cm}
	\begin{align*}
	&\dfrac{1}{2}\|x_{k+1} - x\|^2_{1} \\&  = \dfrac{1}{2}\|x_{k}-x\|^{2}_{1} + L_{i_{k}}(F_{\gamma})\langle x_{k+1}^{(i)} - x_{k}^{(i)}, x_{k}^{(i)} - x^{(i)}  \rangle + \dfrac{L_{i_{k}}(F_{\gamma})}{2}\|x_{k+1}^{(i)} - x_{k}^{(i)}\|^{2}_{(i)} \\
	& = \dfrac{1}{2}\|x_{k}-x\|^{2}_{1} + L_{i_{k}}(F_{\gamma})\langle x_{k+1}^{(i)} - x_{k}^{(i)}, x_{k+1}^{(i)} - x^{(i)}  \rangle - \dfrac{L_{i_{k}}(F_{\gamma})}{2}\|x_{k+1}^{(i)} - x_{k}^{(i)}\|^{2}_{(i)} \\
	&= \dfrac{1}{2} \|x_{k}-x\|^{2}_{1} - \langle U_{i_{k}}^{T}\nabla F_{\gamma}(x_{k}), x_{k+1}^{(i)} - x^{(i)}  \rangle - \dfrac{L_{i_{k}}(F_{\gamma})}{2}\|x_{k+1}^{(i)} - x_{k}^{(i)}\|^{2}_{(i)} \\
	&= \dfrac{1}{2}\|x_{k}-x\|^{2}_{1} - \langle U_{i_{k}}^{T}\nabla F_{\gamma}(x_{k}), x_{k}^{(i)} - x^{(i)}  \rangle \\ 
	& \qquad - \langle U_{i_{k}}^{T}\nabla F_{\gamma}(x_{k}), x_{k+1}^{(i)} - x^{(i)}_{k}  \rangle - \dfrac{L_{i_{k}}(F_{\gamma})}{2}\|x_{k+1}^{(i)} - x_{k}^{(i)}\|^{2}_{(i)} \\
	&\leq \dfrac{1}{2} \|x_{k}-x\|^{2}_{1} - \langle U_{i_{k}}^{T}\nabla F_{\gamma}(x_{k}), x_{k}^{(i)} - x^{(i)}  \rangle + F_{\gamma}(x_{k}) - F_{\gamma}(x_{k+1}). \vspace{-0.2cm}
	\end{align*}
	
	\noindent Taking expectation with respect to $i_{k}$ and using the convexity of $F_{\gamma}$, we obtain
    \vspace{-0.2cm}
	\begin{align*}
	& \dfrac{1}{2} \mathbb{E}_{i_{k}}\left[ \|x_{k+1} - x\|^2_{1} 	\right] \leq \dfrac{1}{2} \|x_{k}-x\|^{2}_{1}  + \dfrac{1}{N} \langle \nabla F_{\gamma}(x_{k}), x-x_{k} \rangle + F_{\gamma}(x_{k}) - \mathbb{E}_{i_{k}}\left[ F_{\gamma}(x_{k+1})\right]  \\ 
	&\leq \dfrac{1}{2}\|x_{k}-x\|^{2}_{1}  + \dfrac{1}{N}(F_{\gamma}(x) - F_{\gamma}(x_{k})) + F_{\gamma}(x_{k}) - \mathbb{E}_{i_{k}}\left[ F_{\gamma}(x_{k+1}) \right]. \vspace{-0.2cm}
	\end{align*}
	
	\noindent Moreover, taking expectation with respect to $i_0,\cdots,i_{k-1}$ and using Assumption \ref{ass1}[A2], with $\mathcal{C}$ being the identity operator, we have
    \vspace{-0.2cm}
	\begin{align*}
	&\dfrac{1}{2}\mathbb{E}\left[ \|x_{k+1} - x\|^2_{1} 	\right] \leq \dfrac{1}{2} \mathbb{E}\left[ \|x_{k}-x\|^{2}_{1}\right]   + \dfrac{1}{N}\mathbb{E}\left[ F_{\gamma}(x) - F_{\gamma}(x_{k})\right]  + \mathbb{E}\left[ F_{\gamma}(x_{k}) -  F_{\gamma}(x_{k+1}) \right] \\
    &\leq \dfrac{1}{2} \mathbb{E}\left[ \|x_{k}-x\|^{2}_{1}\right]   + \dfrac{1}{N}\mathbb{E}\left[ F(x) - F_{\gamma}(x_{k})\right]  + \mathbb{E}\left[ F_{\gamma}(x_{k}) -  F_{\gamma}(x_{k+1}) \right] \\
	&= \dfrac{1}{2} \mathbb{E}\left[ \|x_{k}-x\|^{2}_{1}\right]  + \left( 1- \dfrac{1}{N}\right) \mathbb{E}\left[F_{\gamma}(x_{k}) - F(x)\right]  + \mathbb{E}\left[ F(x) - F_{\gamma}(x_{k+1})\right].  \vspace{-0.2cm}
	\end{align*}
	
	\noindent Hence for $x=\bar{x}_{k}$, we obtain
    \vspace{-0.2cm}
	\begin{align}
	&\dfrac{1}{2} \mathbb{E}\left[ \|x_{k+1} - \bar{x}_{k+1}\|^2_{1} 	\right] + \mathbb{E}\left[ F_{\gamma}(x_{k+1}) - F^{*} \right] \nonumber \\
	&\leq \dfrac{1}{2} \mathbb{E}\left[ \|x_{k}-\bar{x}_{k}\|^{2}_{1}\right]  + \mathbb{E}\left[F_{\gamma}(x_{k}) - F^{*}\right]  - \dfrac{1}{N} \mathbb{E}\left[F_{\gamma}(x_{k}) - F^{*}\right]. \label{eq:5} \vspace{-0.2cm}
	\end{align}

    \noindent First, consider $\bar{q}=2$, from Assumption \ref{ass2}, we have
    \vspace{-0.2cm}
	\begin{align}
	F_{\gamma}(x_{k}) - F^{*} &= \dfrac{\bar\kappa}{1+\bar\kappa} \left( F_{\gamma}(x_{k}) - F^{*} \right) + \left( 1 - \dfrac{\bar\kappa}{1+\bar\kappa}\right)  \left( F_{\gamma}(x_{k}) - F^{*} \right)\nonumber \\
	&\geq \dfrac{\bar\kappa}{1+ \bar\kappa} \left( F_{\gamma}(x_{k}) - F^{*} \right)  + \left( 1 - \dfrac{\bar\kappa}{1+\bar\kappa}\right) \dfrac{\bar\kappa}{2}  \|x_{k}-\bar{x}_{k}\|^{2}_{1}  \nonumber\\
	&= \dfrac{\bar\kappa}{1+ \bar \kappa} \left( F_{\gamma}(x_{k}) - F^{*} + \dfrac{1}{2} \|x_{k}-\bar{x}_{k}\|^{2}_{1}  \right). \label{eq:6}
    \vspace{-0.2cm}
	\end{align} 
	
	\noindent Combining \eqref{eq:5} and \eqref{eq:6}, we obtain 
    \vspace{-0.2cm}
	\begin{align*}
	&\dfrac{1}{2}\mathbb{E}\left[ \|x_{k+1} - \bar{x}_{k+1}\|^2_{1} 	\right] + \mathbb{E}\left[ F_{\gamma}(x_{k+1}) - F^{*} \right] \nonumber \\ 
	&\leq \left( 1 - \dfrac{\bar\kappa}{N(1+\bar\kappa)}\right) \left( \dfrac{1}{2}\mathbb{E}\left[ \|x_{k}-\bar{x}_{k}\|^{2}_{1}\right]  + \mathbb{E}\left[F_{\gamma}(x_{k}) - F^{*}\right] \right). \vspace{-0.2cm}
	\end{align*}
	
	\noindent  Unrolling the recurrence we obtain the statement. 
 
 \medskip 
 
 \noindent Then, consider $\bar q \in [1,2)$
    and $F_{\gamma}(x_{k}) - F^{*} \leq 1$, we have
    \vspace{-0.2cm}
	\begin{align}
	F_{\gamma}(x_{k}) - F^{*} &= \dfrac{\bar\kappa}{1+\bar\kappa} \left( F_{\gamma}(x_{k}) - F^{*} \right) + \left( 1 - \dfrac{\bar\kappa}{1+ \bar\kappa}\right)  \left( F_{\gamma}(x_{k}) - F^{*} \right)\nonumber \\
	&\geq \dfrac{\bar\kappa}{1+\bar\kappa} \left( F_{\gamma}(x_{k}) - F^{*} \right)  + \left( 1 - \dfrac{\bar\kappa}{1+\bar\kappa}\right) \left( F_{\gamma}(x_{k}) - F^{*} \right)^{\frac{2}{\bar{q}}}. \nonumber \vspace{-0.2cm}
	\end{align} 
	
	\noindent On other hand, from Assumption \ref{ass2}, we have 
    \vspace{-0.2cm}
	\begin{equation*}
	\dfrac{\bar{\kappa}^{\frac{2}{\bar q}} }{2} \|x_{k} - \bar{x}_{k} \|_{1}^2 \leq \left( F_{\gamma}(x_{k}) - F^{*} \right)^{\frac{2}{\bar q}}. \vspace{-0.2cm}
	\end{equation*}
	
	\noindent Combining the two inequalities above, we get
    \vspace{-0.2cm}
	\begin{align*}
	F_{\gamma}(x_{k}) - F^{*} &\geq \dfrac{\bar\kappa}{1+ \bar\kappa} \left( F_{\gamma}(x_{k}) - F^{*} \right)  + \left( 1 - \dfrac{\bar\kappa}{1+\bar\kappa}\right) \dfrac{\bar{\kappa}^{\frac{2}{\bar q}} }{2}  \|x_{k} - \bar{x}_{k} \|_{1}^2  \\
	&\geq \dfrac{\eta}{  (1+ \bar\kappa)} \left( F_{\gamma}(x_{k}) - F^{*}   +  \dfrac{1}{2}\|x_{k} - \bar{x}_{k} \|_{1}^2 \right). \vspace{-0.2cm}
	\end{align*}
	
	\noindent Hence, from \eqref{eq:5} and the inequality above, we get
    \vspace{-0.2cm}
	\begin{align}
	&\dfrac{1}{2}\mathbb{E}\left[ \|x_{k+1} - \bar{x}_{k+1}\|^2_{1} 	\right] + \mathbb{E}\left[ F_{\gamma}(x_{k+1}) - F^{*} \right] \nonumber \\ 
	&\leq \left( 1 - \dfrac{\eta}{ N(1+\bar\kappa)}\right) \left( \dfrac{1}{2}\mathbb{E}\left[ \|x_{k}-\bar{x}_{k}\|^{2}_{1}\right]  + \mathbb{E}\left[F_{\gamma}(x_{k}) - F^{*}\right]  \right). \label{eq:20}
    \vspace{-0.2cm}
	\end{align}
	
	\noindent Finally, for $\bar q\in [1,2)$ and $F_{\gamma}(x_{k}) - F^{*} > 1$, using Assumption \ref{ass2}, we get
    \vspace{-0.2cm}
	\begin{align*}
	F_{\gamma}(x_{k}) - F^{*} &= \dfrac{\bar\kappa}{1+\bar\kappa} \left( F_{\gamma}(x_{k}) - F^{*} \right) + \left( 1 - \dfrac{\bar\kappa}{1+\bar\kappa}\right)  \left( F_{\gamma}(x_{k}) - F^{*} \right)\nonumber \\
	&\geq \dfrac{\bar\kappa}{1+\bar\kappa} \left( F_{\gamma}(x_{k}) - F^{*} \right)^{\frac{\bar q}{2}}  + \left( 1 - \dfrac{\bar \kappa}{1+ \bar\kappa}\right) \dfrac{\bar\kappa}{2^{\frac{\bar q}{2}}} \|x_{k}-\bar{x}_{k}\|^{\bar q}_{1}  \nonumber \\
	& = \dfrac{\bar\kappa}{1+\bar\kappa} \left(  \left( F_{\gamma}(x_{k}) - F^{*} \right)^{\frac{\bar q}{2}} +  \dfrac{1}{2^{\frac{\bar q}{2}}}\|x_{k}-\bar{x}_{k}\|^{\bar q}_{1} \right). \vspace{-0.2cm}
	\end{align*} 
	
	\noindent Using the following inequality 
    \vspace{-0.2cm}
	\begin{align*}
	(a + b)^{p} \leq a^{p} +  b^{p} \quad \text{for} \quad a,b \geq 0 \quad  0 <  p < 1,
    \vspace{-0.2cm}
	\end{align*}
	\noindent  in the previous relation for  $p ={\bar q}/2$, we have
    \vspace{-0.2cm}
	\begin{align}
	F_{\gamma}(x_{k}) - F^{*} 
	\geq \dfrac{\bar\kappa}{1+\bar\kappa} \left(  F_{\gamma}(x_{k}) - F^{*}  + \dfrac{1}{2}  \|x_{k}-\bar{x}_{k}\|^{2}_{1} \right)^{\frac{\bar q}{2}}.  \label{eq:34} \vspace{-0.2cm}
	\end{align} 
	
	\noindent Moreover, since $\bar q\in [1,2)$, using \eqref{eq:33} and the fact that $F_{\gamma}(x_{k}) \leq F_{\gamma}(x_{0})$ for all $k \geq 0$ (see (2.4) in \cite{Nes:10}),  we obtain
    \vspace{-0.2cm}
	\begin{align*}
	&\left(  F_{\gamma}(x_{k}) - F^{*}  + \dfrac{1}{2}  \|x_{k}-\bar{x}_{k}\|^{2}_{1} \right)^{\frac{\bar q}{2}} \\
	&= \left(  F_{\gamma}(x_{k}) - F^{*}  + \dfrac{1}{2}  \|x_{k}-\bar{x}_{k}\|^{2}_{1} \right)\left(  F_{\gamma}(x_{k}) - F^{*}  +  \dfrac{1}{2} \|x_{k}-\bar{x}_{k}\|^{2}_{1} \right)^{\frac{\bar q}{2}-1} \\
	&\geq \left(  F_{\gamma}(x_{k}) - F^{*}  + \dfrac{1}{2}  \|x_{k}-\bar{x}_{k}\|^{2}_{1} \right) \left(  F_{\gamma}(x_{0}) - F^{*}  +  R \right)^{\frac{\bar q}{2}-1}. \vspace{-0.2cm}
	\end{align*}
	
	\noindent Hence, from \eqref{eq:33}, \eqref{eq:34} and last inequality, we get
    \vspace{-0.2cm}
	\begin{align*}
	F_{\gamma}(x_{k}) - F^{*} 
	\geq \dfrac{\bar\kappa \Delta_{0}^{ \frac{\bar{q}-2}{2}       }  }{1+\bar\kappa} \left(  F_{\gamma}(x_{k}) - F^{*}  + \dfrac{1}{2}  \|x_{k}-\bar{x}_{k}\|^{2}_{1} \right). \vspace{-0.2cm}
	\end{align*}

	\noindent Combining \eqref{eq:5} and the inequality above, we have	
    \vspace{-0.2cm}
	\begin{align}
	&\mathbb{E}\left[ \|x_{k+1} - \bar{x}_{k+1}\|^2_{1} 	\right] + \mathbb{E}\left[ F_{\gamma}(x_{k+1}) - F^{*} \right] \nonumber \\ 
	&\leq \left( 1 - \dfrac{ \bar\kappa \Delta_{0}^{ \frac{\bar{q}-2}{2}  } }{N(1+\bar\kappa)}\right) \left( \mathbb{E}\left[ \|x_{k}-\bar{x}_{k}\|^{2}_{1}\right]  + \mathbb{E}\left[F_{\gamma}(x_{k}) - F^{*}\right]  \right). \label{eq:21} \vspace{-0.2cm}
	\end{align}
	
	\noindent Hence, from \eqref{eq:20} and \eqref{eq:21}, we get for all $k\geq0$ that 
    \vspace{-0.2cm}
	\begin{align*}
	\mathbb{E}\left[ \|x_{k+1} - \bar{x}_{k+1}\|^2_{1} 	\right] + \mathbb{E}\left[ F_{\gamma}(x_{k+1}) - F^{*} \right]  \leq C_1 \left( \mathbb{E}\left[ \|x_{k}-\bar{x}_{k}\|^{2}_{1}\right]  + \mathbb{E}\left[F_{\gamma}(x_{k}) - F^{*}\right]   \right). \vspace{-0.2cm}
	\end{align*}
	\noindent Unrolling this relation,  the statement follows. 
\end{proof}

\noindent Next theorem presents the convergence rate for the original function

\begin{theorem}
	\label{theo:qdif2F}
	Let Assumptions \ref{ass1} and \ref{ass2} hold with $\mathcal{C}$ being the identity operator. Assuming $i_{k}$ is chosen uniformly at random, then the iterates of the Algorithm 1 satisfy
	\vspace{-0.5cm}
    \begin{align}
	\mathbb{E}\left[ \|x_{k} - \bar{x}_{k}\|^2_{1} 	\right] +  \mathbb{E}\left[ F(\mathcal{B}(x_{k})) -  F^{*}\right] \leq C_1^{k} \left(   \|x_{0}- \bar{x}_{0}\|^{2}_{1}  +    F_{\gamma}(x_{0}) - F^{*}\right) + \gamma D . \vspace{-0.2cm} \label{eq:36}
	\end{align} 
\end{theorem}
\begin{proof}
	By Assumption \ref{ass1}[A.2] we have $\mathbb{E}[F(\mathcal{B}(x_{k})] \leq \mathbb{E}[F_{\gamma}(x_{k})] + \gamma D$. Using Theorem \ref{theo:qdif2}, the result  follows. 
\end{proof}

\noindent Recall that, $\mathcal{C}$ is identity operator for the Moreau envelope, Forward-Backward envelope and Nesterov's smoothing. Hence, in these cases the complexity results are:\\
(i) For Moreau and Forward-Backward envelopes we have $D=0$. Assume that the distances between the points in the level set of $F_{\gamma}$ and their projection onto the optimal set of the original function are bounded. Then,  if $F$ satisfies a $q$-growth condition  with $q\in[1,2]$, from Corollary \ref{cor:1} it follows that also the smooth approximations inherits such a condition, and consequently we get linear convergence, i.e., $\mathcal{O} \left(\ln \left( 1 /\epsilon   \right)   \right)$.  \\
 	\noindent (ii) For Nesterov's smoothing, we need to take $\gamma = \mathcal{O}\left(\dfrac{\epsilon}{2D}   \right)$. Note that from Corollary \ref{cor:1}, if $F_{\gamma}(x_{k-1}) - F^{*} \geq \gamma D$, then $x_{k-1} \in \hat{X}$, where $\hat{X}$ is defined in \eqref{eq:42}. Hence, Assumption \ref{ass2} and
     \eqref{eq:36} hold. On other hand, if $F_{\gamma}(x_{\hat{k}}) - F^{*} < \gamma D$ for some $\hat{k}>0$, from \eqref{ineq2:NS} and \eqref{eq:dec} , we get for all $k \geq \hat{k}$ 
     \vspace{-0.2cm}
	\begin{align*}
	F(x_{k}) - F^{*} \leq F_{\gamma}(x_{k}) - F^{*} + \gamma D \leq F_{\gamma}(x_{\hat{k}}) - F^{*} + \gamma D < 2 \gamma D \leq \mathcal{O}\left(\epsilon \right). \vspace{-0.2cm}
	\end{align*}
	Since $\bar{\kappa} = \mathcal{O}(\gamma^{\frac{\bar q}{2}}   ) = \mathcal{O}(\epsilon^{\frac{\bar q}{2}})$, then in the worst case we have the constant $C_{1} = \mathcal{O} \left(1 - \dfrac{\epsilon}{N(1+ \epsilon)}\right)$ and then the convergence rate is of order 
    $\mathcal{O} \left( \dfrac{1}{\epsilon} \ln \left( \dfrac{1}{\epsilon}   \right)   \right)$.
	

\subsection{Restart accelerated coordinate descent: $q$-growth case}
\noindent In this section we consider Algorithm 2 with $\alpha=0$, i.e., $i_{k}$ is chosen uniformly at random.
\noindent For the sequence $A_k$, we have (see Theorem 1 in \cite{NesSti:17}):
\vspace{-0.2cm}
\begin{equation}
A_k \geq \dfrac{k^2}{4N^2}. \label{eq:15}
\vspace{-0.2cm}
\end{equation}

\noindent Next result was derived  in Theorem 1  in \cite{NesSti:17}. The result was proved for $x_{*} \in \bar{X}^{*}_{\gamma}$, where $\bar{X}^{*}_{\gamma}$ is the set of optimal points of the smooth problem  \eqref{eq:prob2}. However,  the result is also valid for $x^{*} \in X^{*}$, where $X^{*}$ is the set of optimal points of the original problem \eqref{eq:prob}: 
\begin{align}
\label{eq:theoNSACC}
	2 A_{k} \mathbb{E}[F_{\gamma}(x_{k}) - F_{\gamma}(x^{*})] + B_{k} \mathbb{E}[\|\nu_{k} - x^{*}\|_{1}^2] \leq \|x_0 - x^{*}\|_1^2. \vspace{-0.2cm}
	\end{align}

\noindent Let us also introduce a restarting variant of  Algorithm 2 (see  \cite{FerQu:20}):

\begin{center}
	\noindent\fbox{%
		\parbox{8.5 cm}{%
			\textbf{Algorithm 3}:\\
			Choose $x_{0} \in \text{dom} F_{\gamma}$ and set $\tilde{x}_{0} = x_0.$ \\
			Choose restart periods $\{K_0,\cdots,K_r,\cdots\} \subset \mathbb{N}$.\\
			For $r\geq 0$, iterate:
			\begin{enumerate}
				\item $\hat{x}_{r+1} = \text{Algorithm 2} (F_{\gamma},\tilde{x}_r, K_r)             $
				\item $ \tilde{x}_{r+1} = \hat{x}_{r+1} \mathbbm{1}_{F_{\gamma}( \hat{x}_{r+1}) \leq F_{\gamma}(\tilde{x}_{r})} + \tilde{x}_r \mathbbm{1}_{F_{\gamma}(\hat{x}_{r+1})>F_{\gamma}(\tilde{x}_{r})}.$
			\end{enumerate}
	}}
\end{center}

\medskip 

\noindent Following similar arguments as in \cite{FerQu:20}, we get a condition for restarting $x_{k}$, when the original problem has a $q$-growth property,
with $1 \leq q \leq  2$.

\begin{lemma}
	\label{lem:cond}
 Let $(x_{k})_{k\geq0}$ be generated by Algorithm 2 and Assumptions \ref{ass1} and \ref{ass2} hold with $\mathcal{C}$ being the identity operator.  Denote:
    \vspace{-0.2cm}
	\begin{equation*}
	\tilde{x}_{k} = x_{k} \mathbbm{1}_{F_{\gamma}(x_{k}) \leq F_{\gamma}(x_{0})} + x_{0} \mathbbm{1}_{F_{\gamma}(x_{k})>F_{\gamma}(x_{0})}. \vspace{-0.2cm}
	\end{equation*}
	\noindent We have
    \vspace{-0.2cm}
	\begin{align}
	\mathbb{E}[F_{\gamma}(\tilde{x}_{k}) - F^{*}] \leq \dfrac{1}{ A_k \bar{\kappa}^{\frac{2}{\bar q}}}\left(F_{\gamma}(x_{0})-F^{*} \right)^{\frac{2-\bar {q}}{\bar q}} \left(F_{\gamma}(x_{0})-F^{*} \right).  \label{eq:7} \vspace{-0.2cm}
	\end{align}
	
	\noindent Moreover, given $\alpha \in (0,1)$, if 
    \vspace{-0.2cm}
	\begin{align}
	k \geq \sqrt{\dfrac{4 N^2 \left(F(  x_{0})-F^{*} \right)^{\frac{2-\bar{q}}{\bar{q}}} }{\bar{\kappa}^{\frac{2}{q}} \alpha  } },  \label{eq:8} \vspace{-0.2cm}
	\end{align}
	
	\noindent then $\mathbb{E}[F_{\gamma}(\tilde{x}_{k}) - F^{*}] \leq \alpha\left( F_{\gamma}(x_0) - F^* \right) $.
	
\end{lemma}

\begin{proof}
	From \eqref{eq:theoNSACC}, Assumption \ref{ass1}[A2] with $\mathcal{C}$ being the identity operator and Assumption \ref{ass2}, we have 
    \vspace{-0.2cm}
	\begin{align*}
	\mathbb{E}[F_{\gamma}(x_{k}) - F^{*}] &\leq \mathbb{E}[F_{\gamma}(x_{k}) - F_{\gamma}(\bar{x}_{0})] \leq    \dfrac{1}{2 A_k} \|x_{0} - \bar{x}_{0}  \|^2_{1} \\
	&\leq \dfrac{1}{ A_k \bar{\kappa}^{\frac{2}{\bar q}}}\left(F_{\gamma}(x_{0})-F^{*} \right)^{\frac{2}{\bar q}}  = \dfrac{1}{ A_k \bar{\kappa}^{\frac{2}{\bar q}}}\left(F_{\gamma}(x_{0})-F^{*} \right)^{\frac{2-\bar{q}}{\bar q}} \left(F_{\gamma}(x_{0})-F^{*}  \right), \vspace{-0.2cm}
	\end{align*}
 
    \noindent where in the first inequality we used the fact that $F(\bar{x}_{0}) = F^{*}$. Additionally, with probability one $F_{\gamma}(\tilde{x}_{k}) - F^{*} \leq F_{\gamma}(x_{k}) - F^{*}$, which yields \eqref{eq:7}. Moreover, from \eqref{eq:8} and Assumption \ref{ass1}[A2] with $\mathcal{C}$ being the identity,  we get
    \vspace{-0.2cm}
	\begin{equation*}
	\dfrac{4 N^2 }{k^2}  \left( \dfrac{ \left(F_{\gamma}(x_{0})-F^{*} \right)^{\frac{2-\bar{q}}{\bar q}}  }{ \bar{\kappa}^{\frac{2}{\bar q}} } \right) \leq \alpha. \vspace{-0.2cm}
	\end{equation*}
	
	\noindent Hence, from \eqref{eq:15}, \eqref{eq:7} and the inequality above, we obtain
    \vspace{-0.2cm}
	\begin{align*}
	\mathbb{E}[ F_{\gamma}(\tilde{x}_{k}) - F^{*}] \leq \alpha \left( F_{\gamma}(x_{0}) - F^{*}\right). \vspace{-0.5cm}
	\end{align*}
This proves the second statement. 
\end{proof}

\noindent In the next corollary we derive the convergence rate for the restart algorithm when a fixed restart period is considered. 

\begin{corollary} 
\label{cor:rest}
	Let Assumptions \ref{ass1} and \ref{ass2} hold, with $\mathcal{C}$ being the identity operator.  Denote 
    \vspace{-0.2cm}
	\begin{align}
	K(\alpha) = \left \lceil \sqrt{\dfrac{4 N^2 \left(F(x_{0})-F^{*} \right)^{\frac{2-\bar{q}}{\bar q}} }{\bar{\kappa}^{\frac{2}{\bar q}} \alpha  } } \right\rceil. \label{eq:25} \vspace{-0.2cm}
	\end{align}
	If the restart periods $\{K_{0}, \cdots,K_{r}\}$ are all equal to $K(\alpha)$, then the iterates of Algorithm 3 satisfy
    \begin{equation}
    \label{eq:50}
	\mathbb{E}[F_{\gamma}(\tilde{x}_{r})     - F^{*}] \leq \alpha^r\left( F_{\gamma}(x_{0}) - F^{*}\right) 
	\end{equation}
    \noindent and 
	\begin{equation}
    \label{eq:37}
	\mathbb{E}[F(\mathcal{B}(\tilde{x}_{r}))     - F^{*}] \leq \alpha^r\left( F_{\gamma}(x_{0}) - F^{*}\right)  + \gamma D. 
	\end{equation}
\end{corollary} 

\begin{proof}
	By the definition of $\tilde{x}_{r}$, we know that for all $r$, $F(\tilde{x}_{r}) \leq F(x_{0})$. Using Lemma \ref{lem:cond} recursively, \eqref{eq:50} follows. From Assumption \ref{ass1}[A.2] we have $\mathbb{E}[F(\mathcal{B}(\tilde{x}_{r})] \leq \mathbb{E}[F_{\gamma}(\tilde{x}_{r})] + \gamma D$. Using \eqref{eq:50}, the other result also follows. 
\end{proof}

\noindent Recall that $\mathcal{C}$ is identity operator for the Moreau envelope, Forward-Backward envelope and Nesterov's smoothing. Moreover, for the Moreau and Forward-Backward envelopes we obtain a linear rate for Algorithm 3, since $D=0$. On the other hand, in the Nesterov's smoothing we have $D = \bar{D}$. Hence, we need to take $\gamma = \mathcal{O}\left(\dfrac{\epsilon}{2D}   \right)$. Since $\bar{\kappa} = \mathcal{O}(\gamma^{\frac{\bar q}{2}})$, then we have 
$K(\alpha) = \dfrac{1}{ \sqrt{\epsilon}}$. This implies that the convergence rate of  Algorithm 3 in the Nesterov's smoothing is of order $\mathcal{O} \left( \dfrac{1}{ \sqrt{\epsilon}} \ln \left(\dfrac{1}{\epsilon}\right)\right)$. 


\noindent Next we present some conditions to choose the restart periods.

\noindent (i) If $F^{*}\geq 0$, then we can choose the restart periods as
    \vspace{-0.2cm}
	\begin{equation*}
	K(\alpha) = \sqrt{\dfrac{4 N^2 \left(F(x_{0}) \right)^{\frac{2-\bar{q}}{\bar q}} }{\bar{\kappa}^{\frac{2}{\bar q}} \alpha  } }. \vspace{-0.2cm}
	\end{equation*}
 
\noindent 	(ii) If $\bar{\kappa}_{\min}$ is a lower bound for $\bar{\kappa}$, then  we can choose the restart periods \cite{FerQu:20} 
	\begin{equation*}
	K(\alpha) = \sqrt{\dfrac{4 N^2 \left(F(x_{0})-F^{*}  \right)^{\frac{2-\bar{q}}{\bar q}} }{\bar{\kappa}_{\min}^{\frac{2}{\bar q}} \alpha  } }. 
	\end{equation*}
	
\noindent 	(iii) If Assumption \ref{ass:qgrowth} holds with $q\in [1,2]$, then for Moreau envelope, Forward Backward envelope and Nesterov's smoothing, from Corollary \ref{cor:1} we can conclude that if
the distances between the points in the level set of $F_{\gamma}$ and their projections onto the optimal set of the original function are bounded,
then we have that $F_{\gamma}$ satisfies Assumption \ref{ass2} with $\bar{q}=2$. Hence, $K(\alpha)$ becomes
\vspace{-0.2cm}
\begin{align*}
	K(\alpha) = \left \lceil \sqrt{\dfrac{4 N^2}{\bar{\kappa} \alpha  } } \right\rceil. \vspace{-0.5cm}
	\end{align*}
\noindent As showed in \cite{FerQu:20}, if we do not have any knowledge of $\bar \kappa$, the restart periods can be chosen  as follows:
\begin{assumption} Define: 
	\label{assum}
		$K_{0} \in \mathbb{N}^*$; 
		 $K_{2^{j} - 1} = 2^{j} K_0$ for all $j \in \mathbb{N}$;
		 $|\{0\leq r < 2^{J} - 1 | K_{r} = 2^{j}K_{0}\} | = 2^{J-1-j}$ for all $j \in \{0,1,\cdots,J-1\}$ and $J \in \mathbb{N}$.
\end{assumption}
\medskip
\noindent Although next theorem was derived in \cite{FerQu:20} for a different algorithm, it can be also applied for Algorithm 2. Consider $\alpha = \exp(-2)$, then from \eqref{eq:25} we have 
\vspace{-0.2cm}
\begin{align}
\label{eq:27}
K^{*}= K(e^{-2}) = \left \lceil e\sqrt{\dfrac{4 N^2 \left(F(x_{0})-F^{*}    \right)^{\frac{2-\bar{q}}{\bar q}} }{\bar{\kappa}^{\frac{2}{\bar q}}  } } \right\rceil. \vspace{-0.2cm}
\end{align}

\begin{theorem}
	\label{theo:rest}
	Let Assumptions \ref{ass1}, \ref{ass2} and \ref{assum} hold with $\mathcal{C}$ being the identity operator. Then, Algorithm 3 satisfies
    \vspace{-0.2cm}
	\begin{equation}
	\label{eq:28}
	\mathbb{E}[F_{\gamma}(\tilde{x}_{2^J -1}  ) - F^{*}] \leq \bar{\epsilon} \quad \text{and} \quad \mathbb{E}[F(B(\tilde{x}_{2^J -1} )) - F^{*}] \leq \bar{\epsilon} + \gamma D, \vspace{-0.2cm}
	\end{equation}
	where $J= \left\lceil \max(\log_2(K^{*}/K_0), 0) \right\rceil + \left \lceil  \log_2(\ln(\delta_0/ \bar\epsilon)/2) \right\rceil$, 
	$\delta_{0} = F_{\gamma}(x_{0}) - F^{*}$ and $K^{*}$ is defined as in \eqref{eq:27}. To obtain \eqref{eq:28}, the total number of iterations of  Algorithm 2, $K_{0}+\cdots+K_{2^J-1}$, is bounded by 
    \vspace{-0.2cm}
	\begin{equation*}
	\left( \left\lceil \max(\log_2(K^{*}/K_0), 0) \right\rceil + \left \lceil  \log_2(\ln(\delta_0/ \bar\epsilon)) \right\rceil + 1 \right) \ln(\delta_0/\bar\epsilon)\max(K^{*},K_0). \vspace{-0.2cm}
	\end{equation*}
\end{theorem}


\noindent If we do not know the value of $F^{*}$ and $\bar \kappa$, but we know an upper bound for $\bar \kappa$, let us say $\bar\kappa_{\max}$, then we can take in Assumption 4
\vspace{-0.2cm}
\begin{align*}
K_{0} = \left \lceil e \sqrt{\dfrac{4 N^2  |F (x_{0})|^{\frac{2-\bar{q}}{\bar q}}    }{\bar{\kappa}_{\max}^{\frac{2}{\bar q}}  } } \right\rceil. \vspace{-0.2cm}
\end{align*}

\noindent From Theorem \ref{theo:rest}, the complexity of Moreau envelope, Forward-Backward envelope and Nesterov's smoothing
becomes: \\

\noindent (i) Moreau and Forward-Backward envelopes: 
\vspace{-0.2cm}
	\begin{align*}
	\dfrac{1}{ \sqrt{\bar \kappa^{\frac{2}{\bar q}}} } \ln \left( \dfrac{1}{\epsilon}\right) \log_{2} \left( \sqrt{ \dfrac{ \bar \kappa_{\max}^{\frac{2}{\bar q}}}{ \bar \kappa^{\frac{2}{\bar q}}}} \ln \left( \dfrac{F\gamma(x_{0}) - F^{*}}{ \epsilon }  \right)     \right).  \vspace{-0.2cm}
	\end{align*}
	
\noindent (ii) Nesterov's smoothing: 
\vspace{-0.2cm}
	\begin{align*}
	\dfrac{1}{\sqrt{\epsilon} } \ln \left( \dfrac{1}{\epsilon}\right) \log_{2} \left( \sqrt{ \dfrac{ \bar \kappa_{\max}^{\frac{2}{\bar q}}}{ \bar \kappa^{\frac{2}{\bar q}}}} \ln \left( \dfrac{F\gamma(x_{0}) - F^{*}}{ \epsilon }  \right)     \right). \vspace{-0.2cm}
	\end{align*}

\section{Relative smoothness along coordinates}
\label{sec:RelSmooth}
\noindent In the previous sections we  have considered nonsmooth (possibly nonseparable) objective functions $F$. In this section we still assume nonsmooth functions (possibly nonseparable), but relative smooth along coordinates with respect to a given function  (possibly also nonseparable).  Consider $\phi:\mathbb{R}^{n} \to \mathbb{R}$ a strictly convex differentiable function.
Then, the corresponding  Bregman distance is defined as \cite{BauBolTe:16}:
\vspace{-0.2cm}
\begin{align*}
	D_{\phi}(y,x):= \phi(y) - \phi(x) - \langle \nabla \phi(x) , y -x \rangle. 
 \vspace{-0.2cm}
\end{align*}

\noindent Note that $D_{\phi}(y,x) \geq 0$ and $D_{\phi}(y,x) = 0 $ if and only if $y=x$. However, $D_{\phi}(y,x)$ is not necessarily symmetric. Recall that a function $F$ is $L$-smooth (or $F$ has Lipschitz gradient) when it satisfies  Definition  \ref{def:lip} with $N=1$ and $L_1 = L$. As a generalization of a smooth function, papers  \cite{BauBolTe:16} and  \cite{LuFreNes:18} introduced the definition of relative smooth functions. Function $F$ is  smooth relative  to a function $\phi$ if there exists  $L>0$ such that
\vspace{-0.2cm}
\begin{align}
    \label{eq:53}
    F(y) \leq F(x) + \langle \nabla F(x), y-x \rangle + L D_{\phi}(y,x) \quad \forall x,y \in \mathbb{R}^{n}.
\vspace{-0.2cm}
\end{align}

\noindent Note that in relative smooth coordinate descent methods \cite{GaoLu:21,HanRic:21,HiePha:21}, the function $\phi$ is considered separable, i.e, $\phi(x) = \sum_{i=1}^{N} \phi_{i}(x^{(i)})$ and inequality \eqref{eq:53} is defined accordingly, as a generalization of the inequality \eqref{lip:3}. In this section, we define the notion of Bregman distance along coordinates for $\phi$ possibly nonseparable.
\begin{definition}
    Consider $\phi:\mathbb{R}^{n} \to \mathbb{R}$ a differentiable function that is strictly convex along coordinates as given in Definition \ref{def:concor}. We define $D_{\phi}: \mathbb{R}^{n} \to \mathbb{R}$ as the Bregman distance along coordinates associated to the kernel function $\phi$ as follows
    \vspace{-0.2cm}
    \begin{align*}
	D_{\phi}(x+U_{i}d,x):= \phi(x+U_{i}d) - \phi(x) - \langle U_{i}^{T}\nabla \phi(x) , d \rangle \quad \forall i=1,\cdots,N. \vspace{-0.2cm}
\end{align*}
\end{definition}

\noindent Since $\phi$ is strictly convex along coordinates, then $D_{\phi}(x+U_{i}d,x) > 0$ for all $d \neq 0$. Next, we introduce the notion of relative smoothness along coordinates.  

\begin{definition}
	 We say that the function $F:\mathbb{R}^{n} \to \mathbb{R}$ is relative smooth along coordinates with respect to the function $\phi:\mathbb{R}^{n} \to \mathbb{R}$, when it satisfies the following inequality, for all $x \in \mathbb{R}^{n}, d \in \mathbb{R}^{n_{i}}$ and $i=1,\cdots,N$,
    \vspace{-0.2cm}
	\begin{align*}
		F(x+U_{i}d) \leq F(x) + \langle U_{i}^{T}\nabla F(x), d \rangle + L_{i}D_{\phi}(x+U_{i}d,x).    
    \vspace{-0.2cm}
	\end{align*}
\end{definition}

\noindent\red{In this section, we again consider the following probabilities for $i=1,\cdots,N$:
\begin{equation}
\label{probabilities}
p^{(i)}_{\alpha} = \dfrac{L_{i}^{\alpha}}{S_{\alpha}}, \quad \text{with} \quad S_{\alpha} = \sum_{i=1}^{N} L_{j}^{\alpha},
\end{equation}
\noindent and the following norms, for $\alpha\in \mathbb{R}$:
\vspace{-0.2cm}
\begin{equation*}
\|x\|^{2}_{\alpha} =  \sum_{i=1}^{N} L_{i}^{\alpha} \|x^{(i)}\|^{2}_{(i)} \quad \text{and} \quad \left(\|x\|^{*}_{\alpha}\right)^2 =  \sum_{i=1}^{N} L_{i}^{-\alpha} \|x^{(i)}\|^{2}_{(i)}. 
\vspace{-0.2cm}
\end{equation*}}

\noindent Next, we introduce a relative randomized coordinate descent algorithm.
\begin{center}
	\noindent\fbox{%
		\parbox{11.5cm}{%
			\textbf{Algorithm 4 (RRCD):}\\
			Given a starting point $x_{0} \in \rset^n$. \\
			For $k \geq 0$ do: \\
			1. \; Choose $i_{k} \in \{1,...,N\}$ \red{with probability $p_{\alpha}^{(i_{k})}$}. \\
			2. \; Solve the following subproblem:
			\vspace*{-0.2cm}
			\begin{equation}
				\label{eq:subproblem}
				\;\; d_{k} = \arg \min_{d \in \mathbb{R}^{n_{i}}} F(x_{k}) +  \langle U^{T}_{i_k} \nabla F(x_{k}), d \rangle + L_{i_{k}} D_{\phi}\left(x_{k} + U_{i_k}d,x_{k}\right)  .
				\vspace*{-0.2cm}	
			\end{equation} \\
			3. \; 	Update $x_{k+1} = x_{k} + U_{i_k}d_{k}$.
	}}
\end{center}
\medskip
\noindent Throughout this section the following assumptions are valid: 

\begin{assumption}
	\label{ass:relSmooth}
	Assume that: \\
	A.1 $F$ is convex and relative smooth along coordinates with respect to the function $\phi$, with constants $L_{i} > 0$ for all $i=1,\cdots,N$. \\
	A.2 $\phi$ is strictly convex along coordinates, see Definition \ref{def:concor}.\\
    A.3 The gradient of $\phi$ is block coordinate-wise Lipschitz continuous on the level set of $F$, $\mathcal{L}_{F}(x_0)$, with constants $H_{i} > 0$, for all $i=1,\cdots,N$, where $x_{0}$ is the starting point of the algorithm RRCD, see Defintion \ref{def:lip}.\\
    A.4 A solution exists for \eqref{eq:prob} (hence, the optimal value $F^{*} > -\infty$).
\end{assumption}

\noindent Note that in Algorithm RRCD we do not need to know the constants  $H_{i} > 0$. For the proof of next lemma let us define
\vspace{-0.2cm}
\begin{equation}
	\label{eq:Hhmax}
	\Hhmax = \max_{i=1,\cdots,N} H_i. \,\ 
\end{equation}

\begin{lemma}
	\label{lem:relSmo}
	Let the sequence $(x_{k})_{k\geq 0}$ be generated by RRCD and Assumption \ref{ass:relSmooth} hold. Then,
    $x_{k} \in \mathcal{L}_{F}(x_0)$ a.s. for all $k\geq0$ and
    the iterates of RRCD satisfy the descent
	\vspace{-0.2cm}
    \begin{align}
		\mathbb{E}[F(x_{k+1}) \ | \ x_{k} ] &\leq F(x_{k}) - \red{ \dfrac{1}{2S_{\alpha}\Hhmax}  \left(\|\nabla F(x_{k})\|^{*}_{1-\alpha}\right)^2}. \label{eq:14} \vspace{-0.2cm}
	\end{align}
\end{lemma}

\begin{proof}
	\noindent From the optimality condition of the subproblem \eqref{eq:subproblem}, we have
    \vspace{-0.2cm}
	\begin{equation}
		U^{T}_{i_{k}} \nabla F(x_{k}) + L_{i_{k}}U^{T}_{i_{k}}\nabla \phi(x_{k}+U_{i_{k}}d_{k})  - L_{i_{k}}U^{T}_{i_{k}}\nabla \phi(x_{k})  =0.\label{eq:52} \vspace{-0.2cm}
	\end{equation}
 
    \noindent Hence, using Assumption \ref{ass:relSmooth}[A1], we obtain
    \vspace{-0.2cm}
	\begin{align*}
		&F(x_{k+1}) \leq F(x_{k}) +  \langle U^{T}_{i_k} \nabla F(x_{k}), d_{k} \rangle + L_{i_{k}} D_{\phi}\left(x_{k} + U_{i_k}d_{k},x_{k}\right)  \\
		&= F(x_{k}) -  L_{i_{k}}\langle U^{T}_{i_k}\left( \nabla \phi(x_{k}+U_{i_{k}}d_{k}) -  \nabla \phi(x_{k})\right) , d_{k} \rangle + L_{i_{k}} D_{\phi}\left(x_{k} + U_{i_k}d_{k},x_{k}\right).  \vspace{-0.2cm}
	\end{align*}
	
	\noindent Moreover, 
    \vspace{-0.2cm}
	\begin{equation*}
		D_{\phi}\left(x_{k} + U_{i_k}d_{k},x_{k}\right) = \phi(x_{k} +U_{i_{k}}d_{k}) - \phi(x_{k}) -  \langle U^{T}_{i_k}\nabla \phi(x_{k}), d_{k} \rangle. \vspace{-0.2cm}
	\end{equation*}
	
	\noindent Hence,
    \vspace{-0.2cm}
	\begin{align}
		F(x_{k+1}) &\leq F(x_{k}) -  L_{i_{k}} \left(\langle U^{T}_{i_k}\nabla \phi(x_{k}+U_{i_{k}}d_{k}) , d_{k} \rangle -  \phi(x_{k} +U_{i_{k}}d_{k}) +  \phi(x_{k}) \right)\nonumber \\
		&=F(x_{k}) - L_{i_{k}}D_{\phi}\left(x_{k},x_{k} + U_{i_k}d_{k}\right). \label{eq:51} \vspace{-0.2cm}
	\end{align}

    \noindent From Assumption \ref{ass:relSmooth}[A2-A3] and Lemma \ref{lem:1}, we have 
    \vspace{-0.2cm}
    \begin{align*}
        D_{\phi}\left(x_{k},x_{k} + U_{i_k}d_{k} \right) \geq \dfrac{1}{2 H_{i_k}} \| U^{T}_{i_{k}}\nabla \phi(x_{k}+U_{i_{k}}d_{k})  - U^{T}_{i_{k}}\nabla \phi(x_{k})\|^2_{(i_{k})}.
        \vspace{-0.2cm}
    \end{align*}

    \noindent Using \eqref{eq:Hhmax}, \eqref{eq:51} and the inequality above, we get 
    \vspace{-0.2cm}
   \red{\begin{align}
		F(x_{k+1}) &\leq F(x_{k}) -  \dfrac{L_{i_{k}}}{2 \Hhmax } \| U^{T}_{i_{k}}\nabla \phi(x_{k}+U_{i_{k}}d_{k})  - U^{T}_{i_{k}}\nabla \phi(x_{k})\|^2_{(i_{k})}. \label{eq:3} \vspace{-0.2cm}
	\end{align}
	\noindent Hence, $F(x_{k}) \leq F(x_{0})$ for all $k \geq 0$ and the first statement follows. On the other hand, from \eqref{probabilities} we have
    \vspace{-0.2cm}
	\begin{equation*}
		\mathbb{E}\left[ \dfrac{1}{L_{i_{k}}} \| U^{T}_{i_{k}} \nabla F(x_{k})\|^2_{(i_{k})} \ | \ x_{k} \right] = \sum_{i=1}^{N} \dfrac{p_{\alpha}^{(i)}}{L_{i}} \| U^{T}_{i} \nabla F(x_{k})\|^2_{(i)}   = \dfrac{1}{S_{\alpha}} \left(\|\nabla F(x_{k})\|^{*}_{1-\alpha}\right)^2.
	\end{equation*}
	\noindent Hence, from \eqref{eq:52}, we obtain
    \vspace{-0.2cm}
	\begin{align}
		\dfrac{1}{S_{\alpha}} \left(\|\nabla F(x_{k})\|^{*}_{1-\alpha}\right)^2 = \mathbb{E}[L_{i_{k}}\| U^{T}_{i_{k}}\nabla \phi(x_{k}+U_{i_{k}}d_{k})  - U^{T}_{i_{k}}\nabla \phi(x_{k})\|^2_{(i_{k})} \ | \ x_{k} ]. \label{eq:4} \vspace{-0.2cm}
	\end{align}
	\noindent Taking expectation with respect to $x_{k}$ in \eqref{eq:3} and using \eqref{eq:4}, we get \eqref{eq:14}.}
\end{proof}


\subsection{Relative coordinate descent: convex case}
\noindent Our convergence analysis from this section follows similar lines as in \cite{ChoNec:22}. In particular, Lemma \ref{lem:relSmo} implies the inequality (6.1) in \cite{ChoNec:22}, with 
\vspace{-0.2cm}
\red{\begin{align}
\label{eq:C}
C = \dfrac{1}{2S_{\alpha} \Hhmax}. \vspace{-0.2cm}
\end{align}}

\noindent Following the same argument as in \cite{ChoNec:22}, we can get convergence  rates when $F$ is convex.  Define:
\vspace{-0.2cm}
\red{
\begin{equation}
\label{eq:57}
	R_{1-\alpha} = \max_{k \geq 0} \min_{x^{*} \in X^{*}} \|x_{k} - x^{*}\|_{1-\alpha} < \infty \quad \text{and} \quad \Delta_0 = F(x_{0})- F^{*}.
\end{equation} 
}
\begin{theorem}
	\label{the:conv}
	Let the sequence $(x_{k})_{k\geq 0}$ be generated by RRCD. If Assumption \ref{ass:relSmooth} hold, then the following sublinear rate in function values holds:  
    \vspace{-0.2cm}
	\begin{equation}
		\mathbb{E}[F(x_{k}) - F^{*}] \leq \red{\dfrac{2S_{\alpha} \Hhmax R^2_{1-\alpha} }{k + 2S_{\alpha} \Hhmax R^2_{1-\alpha}/\Delta_{0}}. } \label{eq:17} \vspace{-0.2cm}
	\end{equation}
\end{theorem}

\begin{proof}
    See Theorem 7.1 in \cite{ChoNec:22} for a proof.
\end{proof}


\subsection{Relative coordinate descent: $q$-growth case}
In this section we consider 
\vspace{-0.2cm}
\begin{align*}
\bar{X} = \{x \in \mathbb{R}^{n} : 0 < F(x) - F^{*} < r\} \cap \{x \in \mathbb{R}^{n}:  \text{dist}(x, X^{*}) \leq \gamma \}, \vspace{-0.2cm}   
\end{align*}
\red{for some $r,\gamma> 0$}. We assume that the function $F$ satisfies Assumption \ref{ass:qgrowth} with $X=\bar{X}$ and $q\in[1,2]$. Since $F$ is convex, we have that Assumption \ref{ass:qgrowth} is equivalent to \eqref{KL}, see \cite{BolNgu:16}. In algorithm RRCD, \red{$i_{k}$ is chosen with the probability $p_{\alpha}^{(i_{k})}$} and the inequality \eqref{KL} is valid only in a neighborhood. Using similar arguments as in \cite{MauFadAtt:22,ChoNec:22} one can show that there exists an iteration $\bar{k} > 0$ such that inequality \eqref{KL} holds for all $k \geq \bar{k}$, with some probability.
Next lemma presents some basic properties for  $X(x_{0})$,  the limit points of the sequence $(x_{k})_{k\geq 0}$. 

\begin{lemma}
	\label{LemmaKL}
	Let the sequence $(x_{k})_{k\geq 0}$  generated by RRCD be bounded. If  Assumption \ref{ass:relSmooth} holds, then  $X(x_{0})$ is  a compact set,  $F (X(x_{0})) = F^*$,  $F(x_{k})  \to F^{*}$ a.s., and $\nabla F(X(x_{0}))=0$,   \red{ $\| \nabla F(x_{k})\|_{1-\alpha}^{*}  \to 0$ a.s.}
\end{lemma}

\begin{proof}
    See Lemma 6.4 in \cite{ChoNec:22} for a proof. 
\end{proof}

\noindent Next theorem presents the convergence rate for the relative smooth algorithm RRCD when the function $F$ satisfies the $q$-growth assumption with $q\in[1,2]$ (i.e., Assumption \ref{ass:qgrowth}). Consider  \red{$C_2=C\mu_{\alpha}^2\Delta_0^{2p-1}$} and \red{$C_3= \mu_{\alpha}^2 \Delta_0^{2p}C$}, with $p=\frac{q-1}{q}$, $C$ defined in \eqref{eq:C}, $\Delta_{0}$ defined in \eqref{eq:57} and \red{$\mu_{\alpha}>0$ being a constant such that
\vspace{-0.2cm}
\begin{equation}
    \label{KL:alpha}
    \mu_{\alpha}(F(x) - F^{*})^{p} \leq \|\nabla F(x)\|_{1-\alpha}^* \quad \forall x \in \text{dist}(x, X^{*}) \leq \gamma, F^{*} < F(x)  < F^{*} + r. \vspace{-0.3cm}
\end{equation}}

\begin{theorem}
    \label{the:KL}
	Let the sequence $(x_{k})_{k\geq 0}$ generated by RRCD be bounded. If Assumptions \ref{ass:qgrowth} and \ref{ass:relSmooth} hold, with $X=\bar{X}$, then for any $ \delta > 0$ there exist a measurable set $\Omega_{ \delta}$ satisfying  $ \mathbb{P}[\Omega_{ \delta}] \geq 1 - \delta$ and $k_{ \delta,\gamma,r} > 0$ such that with probability at least $1 - \delta$ the following statements hold, for all $k\geq k_{\delta,\gamma,r}$: 
    \vspace{-0.2cm}
	\begin{equation}
        \label{ineq:KL}
		\mathbb{E}[F(x_{k}) - F^{*}] \leq \left( 1 - C_{2}  \right)^{k-k_{\delta,\gamma,r}} \left( F(x_{k_{\delta,\gamma,r}}) - F^{*} + C_{3}\sqrt{\delta}  \right) + C_{3}\sqrt{\delta}. \vspace{-0.2cm}
	\end{equation}
\end{theorem}

\begin{proof}
    Since the function $F$ is convex,  Assumption \ref{ass:qgrowth}  with $X=\bar{X}$ \red{is equivalent to \eqref{KL} with $p=\frac{q-1}{q}$ (see \cite{BolNgu:16}). Moreover, \eqref{KL} is equivalent to \eqref{KL:alpha}.} 
	From Lemma \ref{LemmaKL}, we have that $F(x_{k}) \inas F^*$ and \red{$\|\nabla F(x_{k}) \|_{1-\alpha}^{*} \inas 0$},  i.e., there exists a set $\Omega$ such that  $ \mathbb{P}[\Omega]  = 1$ and for all $\omega \in \Omega: F(x_{k}(\omega)) \to F^*(\omega)$ and \red{$\|\nabla F(x_{k}(\omega)) \|_{1-\alpha}^{*} \to 0$}.	
Moreover, from the Egorov’s theorem (see  \red{\cite[Theorem 4.4]{SteSha:05}}), we have that for any  $ \delta>0$ there exists a measurable set $\Omega_{ \delta} \subset \Omega$ satisfying $ \mathbb{P}[\Omega_{ \delta}] \geq 1 - \delta$ such that  $F(x_{k})$ converges uniformly to $F^{*}$ and $\nabla F(x_{k})$  converges uniformly to $0$  on the set $\Omega_{ \delta}$.
Since $F$ satisfies inequality \eqref{KL:alpha},  there exists a $k_{ \delta,\gamma,r} > 0$ and $\Omega_{\delta} \subset \Omega$ with $\mathbb{P}[\Omega_{\delta}] \geq 1 - \delta$ such that  $\text{dist}(x_k(\omega), X(x_{0})) \leq \gamma, \; F^* < F(x_{k}(\omega)) < F^* + r$ for all $k \geq k_{\delta,\gamma,r}$ and $\omega \in \Omega_{ \delta}$, and additionally: 
    \vspace{-0.2cm}
	\red{\begin{equation}
		\label{eq:139}
		\mu_{\alpha} (F(x_{k}(\omega)) - F^*)^p  \leq  \|\nabla F(x_{k}(\omega))\|_{1-\alpha}^{*} \quad \forall k \geq k_{\delta,\gamma,r} \text{ and } \omega \in \Omega_{ \delta}. \vspace{-0.2cm}
	\end{equation}}
	\noindent Equivalently, we have:
    \vspace{-0.2cm}
	\red{\begin{equation*}
	  \mathbbm{1}_{\Omega_{ \delta}}	\left( \mu^2_{\alpha} (F(x_{k}) - F^*)^{2p}  \right) \leq  \mathbbm{1}_{\Omega_{ \delta}} \left( \|\nabla F(x_{k})\|^{*}_{1-\alpha}\right)^2 \quad \forall k \geq k_{\delta,\gamma,r}. \vspace{-0.2cm}
	\end{equation*}}
 
	\noindent Taking expectation on both sides of the previous inequality and since $F(x_{k}) \leq F(x_{0})$, from \red{\cite[Lemma 9.1]{ChoNec:22}} we have for all $k \geq k_{\delta,\gamma,r}$:
    \vspace{-0.2cm}
	\begin{align*}
		 \mu_{\alpha}^2\mathbb{E}[ (F(x_{k}) - F^*)^{2p}] - \red{\mu_{\alpha}^2 \Delta_0^{2p} \sqrt{\delta}} &\leq 
		\mathbb{E}[\mathbbm{1}_{\Omega_{ \delta}}	\left( \mu_{\alpha}^2 (F(x_{k}) - F^*)^{2p} \right) ] \\
		&\red{ \leq  \mathbb{E}\left[\mathbbm{1}_{\Omega_{ \delta}}  \left( \|\nabla F(x_{k})\|^{*}_{1-\alpha}\right)^2 \right] \leq \mathbb{E}\left[  \left( \|\nabla F(x_{k})\|^{*}_{1-\alpha}\right)^2 \right].}  \vspace{-0.2cm}
	\end{align*}

\noindent Taking expectation in the inequality \eqref{eq:14}, w.r.t. $\{x_{0}, \cdots, x_{k-1}\}$, and combining with the inequality above, we get:
\vspace{-0.2cm}
\begin{align*}
		 C\left(\mu_{\alpha}^2\mathbb{E}[ (F(x_{k}) - F^*)^{2p}] - \red{\mu_{\alpha}^2 \Delta_0^{2p} \sqrt{\delta}} \right) \leq 
		 \mathbb{E}[ (F(x_{k}) - F^*)]  - \mathbb{E}[ (F(x_{k+1}) - F^*)] .\vspace{-0.2cm}
\end{align*}

\noindent Since $(F(x_{k}))_{k\geq 0}$ is decreasing and $p \in [0,1/2]$, we have that 
\vspace{-0.2cm}
        \begin{align*}
		 C\left(\mu_{\alpha}^2\mathbb{E}[ (F(x_{k}) - F^*)] \Delta_0^{2p-1} - \red{\mu_{\alpha}^2 \Delta_0^{2p} \sqrt{\delta}}  \right) \leq 
		 \mathbb{E}[ (F(x_{k}) - F^*)]  - \mathbb{E}[ (F(x_{k+1}) - F^*)]. \vspace{-0.2cm}
          \end{align*}

\noindent Hence, 
\vspace{-0.2cm}
        \begin{align}
        \label{eq:29}
		 C_2\mathbb{E}[ (F(x_{k}) - F^*)] - C_{3} \sqrt{\delta} \leq 
		 \mathbb{E}[ (F(x_{k}) - F^*)]  - \mathbb{E}[ (F(x_{k+1}) - F^*)].  \vspace{-0.2cm}
\end{align}

\noindent Note that, if   $ \mathbb{E}[F(x_{\bar{k}}) - F^*] \leq C_{3}\sqrt{\delta}$ for some $\bar{k} \geq  k_{\delta,\gamma,r}$, then the statement holds for $k \geq \bar{k}$, since $(F(x_{k}))_{k\geq 0}$ is decreasing. Otherwise, if $ \mathbb{E}[F(x_{k+1}) - F^*] > C_{3}\sqrt{\delta}$, from \eqref{eq:29}, we obtain
\vspace{-0.2cm}
\begin{align*}
		\mathbb{E}[ (F(x_{k+1}) - F^*)] - C_{3} \sqrt{\delta} \leq (1-C_2)\left(\mathbb{E}[ (F(x_{k}) - F^*)] - C_{3} \sqrt{\delta} \right). \vspace{-0.2cm} 
\end{align*}

\noindent Unrolling the sequence, we get  the linear rate. 
\end{proof}

\subsection{Implementation details in relative smooth case}
Some examples of relative smooth functions and their constant are given in \cite{LuFreNes:18}. In this section we derive the relative smooth constant along coordinates of a function presented in \cite{LuFreNes:18}. Moreover, we present an explicit solution for the subproblem in RRCD, for two examples of functions $\phi$. 

\medskip

\noindent \textbf{Example of a function which is relative smooth along coordinates:} \red{Consider the function}
$G(x) = \dfrac{1}{4} \|Ex\|^4 + \dfrac{1}{2} \|Ax-b\|_{\ell_4}^{4} + \dfrac{1}{2} \|Cx-d\|^2$. In \cite{LuFreNes:18} it was proved that this function $G$ is relative smooth w.r.t. the function $\phi(x) = \dfrac{1}{4}\|x\|^4 + \dfrac{1}{2}\|x\|^2$. We consider a more general function defined as $F(x) =  f(x) + \dfrac{1}{4} \|Ex\|^4 + \dfrac{1}{2} \|Ax-b\|_{\ell_4}^{4} $, where $f$ is coordinate-wise Lipschitz with constants $L_{i}(f)>0$ for $i=1,\cdots,N$, see Defintion \ref{def:lip}. 
Let us derive the relative smooth constants along coordinates of $F$. Consider $D(x) = \text{diag}(Ax-b)$, then we have
\begin{align*}
	U_{i}^{T}\nabla^2F(x)U_{i} &= U_{i}^{T}\nabla^2f(x)U_{i} + \|Ex\|^2 (EU_{i})^{T}(EU_{i}) \\ &+ \red{2 (EU_{i})^{T}Exx^{T}E^{T}EU_{i}}
	+ 3 (AU_{i})^{T}D^2(x)AU_i .
\end{align*}

\noindent Hence,
\begin{align*}
	& \|U_{i}^{T}\nabla^2F(x)U_{i}\| \leq \|U_{i}^{T}\nabla^2f(x)U_{i}\| + \|EU_{i}\|^2 \|E\|^2 \|x\|^2 \\ &+ 2\|EU_{i}\|^2 \|E\|^2 \|x\|^2 + 3\|AU_{i}\|^2(\|b\| + \|A\| \|x\|)^2 \\
	&= \left( L_{i}(f) +  3\|AU_{i}\|^2 \|b\|^2 \right) \!+\! 6\|AU_{i}\|^2\|b\|\|A\| \|x\| \!+\! \left( 3\|EU_{i}\|^2 \|E\|^2 \!+\! 3\|AU_{i}\|^2 \|A\|^2\right) \|x\|^2.
\end{align*}

\noindent Since 
\begin{align*}
	U_{i}^{T}\nabla^2 \phi(x)U_{i} = 2 U_{i}^{T}xx^{T}U_{i} + (\|x\|^2  + 1)  I_{n_{i}}  \succeq (\|x\|^2  + 1) I_{n_{i}}, 
\end{align*}

\noindent  we have 
\begin{align*}
	L_{i} &= L_{i}(f) + 3\|AU_{i}\|^2 \|b\|^2   + 6\|AU_{i}\|^2\|b\|\|A\| + 3\|EU_{i}\|^2 \|E\|^2 + 3\|AU_{i}\|^2 \|A\|^2 \\
	&= L_{i}(f) + 3 \|AU_i\|^2 \left(\|b\| + \|A\|    \right)^2 + 3 \|EU_{i}\|^2 \|E\| .
\end{align*}

\medskip

 \noindent Hence,  we can use algorithm  RRCD for minimizing this function $F$ and Step 2 of the algorithm RRCD can be efficiently implemented, i.e., in closed form as we will show below. On the other hand, if we consider $F=f+\psi$, where $f$ has coordinate-wise Lipschitz gradient and  $\psi(x)=\dfrac{1}{4} \|Ex\|^4 + \dfrac{1}{2} \|Ax-b\|_{\ell_4}^{4}$, then one can use the algorithms from \cite{ChoNec:22,NecCho:21} for minimizing $F$. However, in these algorithms one needs to solve a subproblem of the form: 
\vspace{-0.2cm}
\begin{equation}
	\label{eq:subproblem2}
	\;\; d_{k} = \arg \min_{d \in \mathbb{R}^{n_{i}}} f(x_{k}) +  \langle U^{T}_{i_k} \nabla f(x_{k}), d \rangle + \dfrac{L_{i_{k}}}{2} \|d\|^2 + \psi(x_{k} + U_{i_k}d),
	\vspace*{-0.1cm}	
\end{equation}

\noindent which does not have a closed form solution and it is more difficult to solve it.

\medskip

\noindent \textbf{Efficient solution for subproblem:} 
Note that solving \eqref{eq:subproblem} is equivalent to minimizing
\vspace{-0.2cm}
\begin{align}
	d^{*} = \arg \min_{d \in \mathbb{R}^{n_{i}}}  \langle U^{T}_{i} \nabla F(x) - L_{i} U^{T}_{i} \nabla \phi(x) , d \rangle + L_{i} \phi(x+U_id). \label{eq:subp} \vspace{-0.3cm}
\end{align}

\noindent Next, we show that if a function $F$ is relative smooth along coordinates w.r.t. $\phi_1(x) = \dfrac{1}{p} \|x\|^p + \dfrac{1}{2}\|x\|^2$, with $p> 2$, or $\phi_2(x) = \dfrac{1}{2} \langle Ax,x\rangle$, then the subproblem of RRCD can be solved in  closed form. Moreover, if the level set of the function $F$, $\mathcal{L}_{F}(x_0)$, is bounded, then Assumption \ref{ass:relSmooth}[A3] holds for $\phi_1$ and $\phi_2$, since both are twice differentiable.

\medskip

\noindent 1. Consider $\phi(x) = \dfrac{1}{p}\|x\|^p + \dfrac{1}{2}\|x\|^2$ with $p> 2$. In this case we have,
$\nabla \phi(x) = \|x\|^{p-2}x + x$. Moreover, the subproblem \eqref{eq:subp} becomes
\vspace{-0.2cm}
\begin{align*}
	d^{*} = \arg \min_{d \in \mathbb{R}^{n_{i}}} \langle U^{T}_{i} \nabla F(x) - L_{i} \|x\|^{p-2}U_{i}^{T}x , d \rangle + \dfrac{L_{i}}{2} \|d\|^2 + \dfrac{L_{i}}{p} \|x+U_id\|^p. \vspace{-0.2cm}
\end{align*}

\noindent Denoting $g_{i}(x) := U^{T}_{i} \nabla F(x) - L_{i} \|x\|^{p-2}U_{i}^{T}x  $, we have 
\vspace{-0.2cm}
\begin{align*}
	d^{*} = \arg \min_{d \in \mathbb{R}^{n_{i}}} \langle g_{i}(x), d \rangle + \dfrac{L_{i}}{2} \|d\|^2 + \dfrac{L_{i}}{p} \|x+U_id\|^p. \vspace{-0.2cm}
\end{align*}

\noindent From the optimality condition, we have
\vspace{-0.2cm}
\begin{align}
	&g_{i}(x) + L_i d^{*} + L_i \|x+U_i d^{*}\|^{p-2} U_{i}^{T} (x+ U_{i}d^{*}) = 0 \nonumber \\
	&\Leftrightarrow (L_i + L_i \|x+U_i d^{*}\|^{p-2}) d^{*} = -g_{i}(x) - L_i\|x+U_i d^*\|^{p-2}U_{i}^{T}x. \label{eq:24}
\end{align}

\noindent Note that if the value of $\|x+U_i d^*\|$ is known, then $d^{*}$ can be computed as:
\vspace{-0.2cm}
\begin{align*}
	d^{*} = \dfrac{-g_{i}(x) - L_i\|x+U_i d^*\|^{p-2}U_{i}^{T}x }{L_i(1 + \|x+U_i d^*\|^{p-2})}. \vspace{-0.2cm}
\end{align*}

\noindent Next, in order to get the value of $\|x+U_i d^*\|$ we need to find a positive root of a polynomial equation. Indeed, multiplying the equality \eqref{eq:24} by $U_{i}$ and adding  $\left(L_{i}+\|x+U_i d^*\|^{p-2}\right)x$ on both sides of the inequality, we have 
\vspace{-0.2cm}
\begin{align*}
	&(L_i + L_i \|x+U_i d^{*}\|^{p-2}) (x+ U_{i}d^{*}) \\
	&= -U_{i}g_{i}(x) - L_i\|x+U_i d^*\|^{p-2}U_{i}U_{i}^{T}x + \left(L_{i}+\|x+U_i d^*\|^{p-2}\right)x. \vspace{-0.2cm}
\end{align*}

\noindent Taking the norm on both sides of the equality and denoting $\alpha = \|x+U_i d\|$ we have
\vspace{-0.2cm}
\begin{align*}
	&L_i\alpha + L_i \alpha^{p-1} 
	=\| U_{i}g_{i}(x) + L_i\alpha^{p-2}(U_{i}U_{i}^{T}x - x) -L_{i}x\| \\
	& \Leftrightarrow ( L_i\alpha + L_i \alpha^{p-1} )^2 = \| U_{i}g_{i}(x) + L_i\alpha^{p-2}(U_{i}U_{i}^{T}x - x) -L_{i}x\|^2 \\
	& \Leftrightarrow L_i^2 \alpha^{2(p-1)} + 2L_{i}^2 \alpha^{p} + L_{i}^2\alpha^2 = \|g_{i}(x) - L_{i}U^{T}_i x\|^2 + (L_{i}\alpha^{(p-2)} + L_{i})^2 \sum_{j\neq i} \|x^{(j)} \|^2 \\
	& \Leftrightarrow L_i^2 \alpha^{2(p-1)} + 2L_{i}^2 \alpha^{p} + L_{i}^2\alpha^2 \\
	&\qquad = \|g_{i}(x) - L_{i}U^{T}_i x\|^2 + (L_{i}^2\alpha^{2(p-2)} + 2L_{i}^2\alpha^{(p-2)}+ L_{i}^2) \sum_{j\neq i} \|x^{(j)} \|^2\\
	& \Leftrightarrow L_i^2 \alpha^{2(p-1)} - L_{i}^2\alpha^{2(p-2)} \sum_{j\neq i} \|x^{(j)} \|^2
	+ 2L_{i}^2 \alpha^{p} - 2L_{i}^2\alpha^{(p-2)} \sum_{j\neq i} \|x^{(j)} \|^2 + L_{i}^2\alpha^2  \\
	& \qquad  = \|g_{i}(x) - L_{i}U^{T}_i x\|^2 + L_{i}^2 \sum_{j\neq i} \|x^{(j)} \|^2. \vspace{-0.2cm}
\end{align*}

\noindent Hence, the value of $\|x+U_id\|$ is a nonnegative root of the polynomial
\vspace{-0.1cm}
\begin{align*}
	a\alpha^{2(p-1)} -  b\alpha^{2(p-2)} + 2a \alpha^{p} - 2b\alpha^{(p-2)} + a \alpha^{2} - \red{c} = 0, \vspace{-0.2cm}
\end{align*}

\noindent with 
\vspace{-0.2cm}
\begin{align*}
	&a =  L_i^2, \,\ b = L_{i}^2\sum_{j\neq i} \|x^{(j)} \|^2 \text{ and } \red{c} = \|U_{i}g_{i}(x) - L_{i}U^{T}_i x\|^2 + L_{i}^2 \sum_{j\neq i} \|x^{(j)} \|^2. \vspace{-0.2cm}
 \end{align*}

\noindent In particular, for $p=4$ (as in the previous example) we have 
\vspace{-0.2cm}
\begin{align}
	a\alpha^{6} + (2a-b)\alpha^{4} + ( a - 2b)\alpha^{2} - \red{c} = 0. \label{eq:pol4} \vspace{-0.2cm}
\end{align}

\noindent Note that the polynomial above has only one change of sign when $\red{c}\neq0$. 
Then, using Descarte's rule of signs  we have that \eqref{eq:pol4} has only one positive root. Moreover, if $\red{c}=0$, then $\alpha =0$ is a root of \eqref{eq:pol4}.

\medskip

\noindent 2. Now consider $\phi(x) = \langle Ax,x\rangle$. Let us consider the scalar case, i.e., $U_{i} = e_{i}$. Then, $\nabla \phi(x) = Ax$ and the subproblem \eqref{eq:subp} becomes
\vspace{-0.2cm}
\begin{align*}
	d^{*} = \arg \min_{d \in \mathbb{R}}  \langle e^{T}_{i} \nabla F(x) - L_{i} e^{T}_{i}Ax , d \rangle + L_{i} \langle A(x+e_id), x+e_id\rangle.
 \vspace{-0.2cm}
\end{align*}

\noindent Hence, from the optimality conditions we have
\vspace{-0.2cm}
\begin{align*}
	e^{T}_{i} \nabla F(x) - L_{i} e^{T}_{i}Ax + L_{i} e^{T}_{i}A(x+e_id^{*}) = 0. \vspace{-0.2cm}
\end{align*}

\noindent Then,
\vspace{-0.2cm}
\begin{align*}
	d^{*} = \dfrac{- e^{T}_{i} \nabla F(x)}{L_{i} A_{ii}}, \vspace{-0.2cm}
\end{align*}

\noindent where $A_{ii}$ is the $i$th component of the diagonal of A.


\section{Simulations}
\label{sec:Simulations}
In this section we consider two applications. In the first problem the objective is quadratic with an $\ell_2$ regularization. For this application we test the performance of the smoothing techniques in the convex case.  In the second problem we consider a quadratic objective with TV regularization and  test the performance of the smoothing techniques in the $q$-growth case. 
All the implementations are in Matlab and only one component is updated at each iteration, i.e., $U_{i}=e_{i}$ in all the simulations. 

\subsection{Quadratic objective with $\ell_2$ regularization}
We consider the problem
\begin{equation*}
\min_{x \in \mathbb{R}^{n}} \dfrac{1}{2}\|Bx- c\|^2 + \lambda \|x\| 
\end{equation*}

\noindent with $B \in \mathbb{R}^{m\times n}$ and $c\in\mathbb{R}^{m}$.
We apply  Algorithms 1 and 2 using Moreau envelope, Forward-Backward envelope, Douglas-Rachford envelope and Nesterov's smoothing. Matrix $B \in \mathbb{R}^{m \times n}$ is generated sparse from a normal distribution $\mathcal{N}(0,1)$. Vector $c \in \mathbb{R}^{m}$ and the starting point $x_{0} \in \mathbb{R}^{n}$ are also generated  from a normal distribution $\mathcal{N}(0,1)$. We compared the following algorithms: \\
1. ME-CD: Moreau envelope using Algorithm 1. \\
2. ME-ACC: Moreau envelope using Algorithm 2. \\
3. FB-CD: Forward-Backward envelope using Algorithm 1. \\
4. FB-ACC: Forward-Backward envelope using Algorithm 2. \\
5. DR-CD: Douglas-Rachford envelope using Algorithm 1. \\
6. DR-ACC: Douglas-Rachford envelope using Algorithm 2. \\
7. NS-CD: Nesterov's smoothing using Algorithm 1. \\
8. NS-ACC: Nesterov's smoothing using Algorithm 2. 

\noindent We stop the algorithms when
\vspace{-0.2cm}
\begin{equation*}
\|\nabla F_{\gamma}(x)\| \leq 10^{-1}, \vspace{-0.2cm}
\end{equation*}
\noindent and we report in Table 2 the full iterations and the cpu time in seconds for each method. Moreover, we plot the function values along time (in seconds) in  Figure 1. \red{Note that since $m<n$, then the problem is convex, but the objective function is  not strongly convex and thus the $2$-growth condition does not hold in this case.}
For Forward-Backward envelope, Douglas Rachford envelope and Nesterov's smoothing, we have an explicit expression for $e_{i}^{T} \nabla F_{\gamma}(x)$. However, for the Moreau envelope, we do not have a explicit solution for the proximal operator and in this case  we use CVX to compute the prox.  As expected, the results from Table \ref{table:1} and Figure 1 show that the accelerated coordinate descent algorithm has better performance compared to the non-accelerated variant. Moreover, as we can see from Table \ref{table:1}, Douglas-Rachford had a better performance in the norm of the gradient, while in the function values  Forward-Backward and Douglas-Rachford are comparable (see Figure 1). Also  Nesterov's smoothing has a better performance when $\lambda$ is small.
\begin{table}[h!]
	\centering    
	\begin{tabular}{| c| c| c| c | c | c | c| c | c | c |}
		\hline
		\text{n} & 100 & 100 & 100 & $10^3$ & $10^3$ & $10^3$ & $10^4$ & $10^4$ & $10^4$ \\ \hline
		m       &   50&   50    & 50   & 500        & 500      &   500 & $5 \cdot 10^3$ & $5 \cdot 10^3$ & $5 \cdot 10^3$     \\ \hline
		$\lambda$ & 1 & 0.5 & 0.1  & 1 & 0.5 & 0.1 & 1 & 0.5 & 0.1  \\ \hline
		\multirow{2}{*}{ME-CD} & 874  & 1265 & 244 & & & &  &  &  \\ 
		& 18095  & 25900 & 5013 & ** & ** & ** & ** & ** & ** \\ \hline
		\multirow{2}{*}{ME-ACC} & 119  & 108 & 68 & & & & &  &  \\ 
		&  2472 & 2238 & 1411 & ** & ** & ** & ** & ** & ** \\  \hline
		\multirow{2}{*}{FB-CD} & 802  & 1204 & 227 & 3585 & 6001 & 1240  & & & 4070  \\ 
		& 0.802  & 1.22 & 0.227 & 110 & 184 & 37.3 & ** & ** & 18454  \\ \hline
		\multirow{2}{*}{FB-ACC} & 160  & 99 & 104  & 248 & 209 & 249  & 475 & 382 & 423  \\ 
		& 0.228  & \textbf{0.144} & \textbf{0.149} & 8.96 & 7.57 & 8.98 & 2388 & 1831 & 2112 \\  \hline
		\multirow{2}{*}{DR-CD} & 1510  & 2336 & 453  & 6732 & 11303 & 2328  &  &  & 7691   \\ 
		& 0.926  & 1.38 & 0.285 & 71.5 & 115 & 24.7 & ** & ** & 11196  \\  \hline
		\multirow{2}{*}{DR-ACC} & 215  & 204 & 162 & 456 & 290 & 359  & 658 & 536 & 604  \\ 
		& \textbf{0.223}  & 0.216 & 0.167 & \textbf{8.77} & \textbf{5.8} & 7.05 & \textbf{1600} & \textbf{1248} & 1472  \\  \hline
		\multirow{2}{*}{NS-CD} & 12734  & 10985 & 440  & 43253 & 35749 & 1364  & & & 4739 \\ 
		&  7.86 & 6.84 & 0.265 & 337 & 284 & 10.4 & ** & ** & 2031 \\  \hline
		\multirow{2}{*}{NS-ACC} & 653 & 449 & 151  & 1175 & 801 & 254  & 2104 & 1424 & 459  \\ 
		&  0.668 & 0.477 & 0.155 & 15.7 & 11 & \textbf{3.3}  & 2131 & 1307 & \textbf{474} \\  \hline
	\end{tabular}
		\caption{Comparison of Algorithms 1 and 2 for four different smoothing techniques and several problem dimensions in terms of number of full iterations and  cpu time (sec). }
 \label{table:1}
\end{table}


\begin{figure}[ht!]
\label{funcL2}
	\centering
	\includegraphics[width=6.3cm]{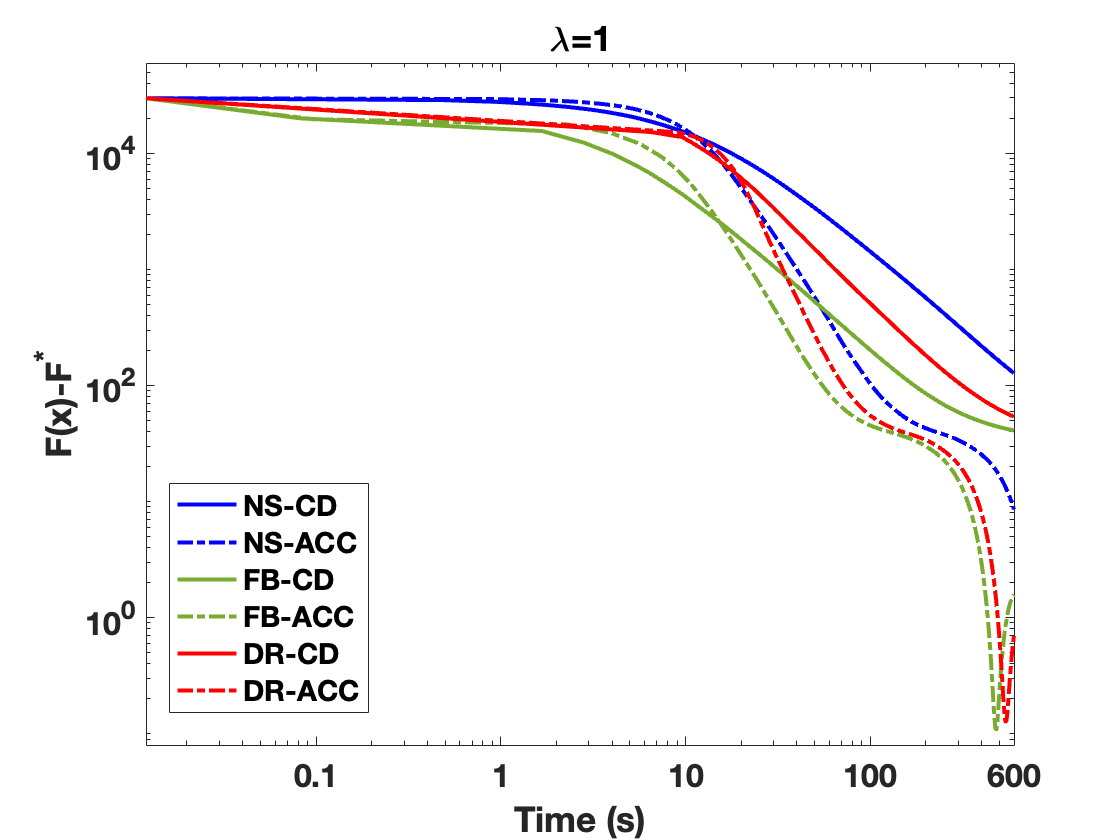} 
	\includegraphics[width=6.3cm]{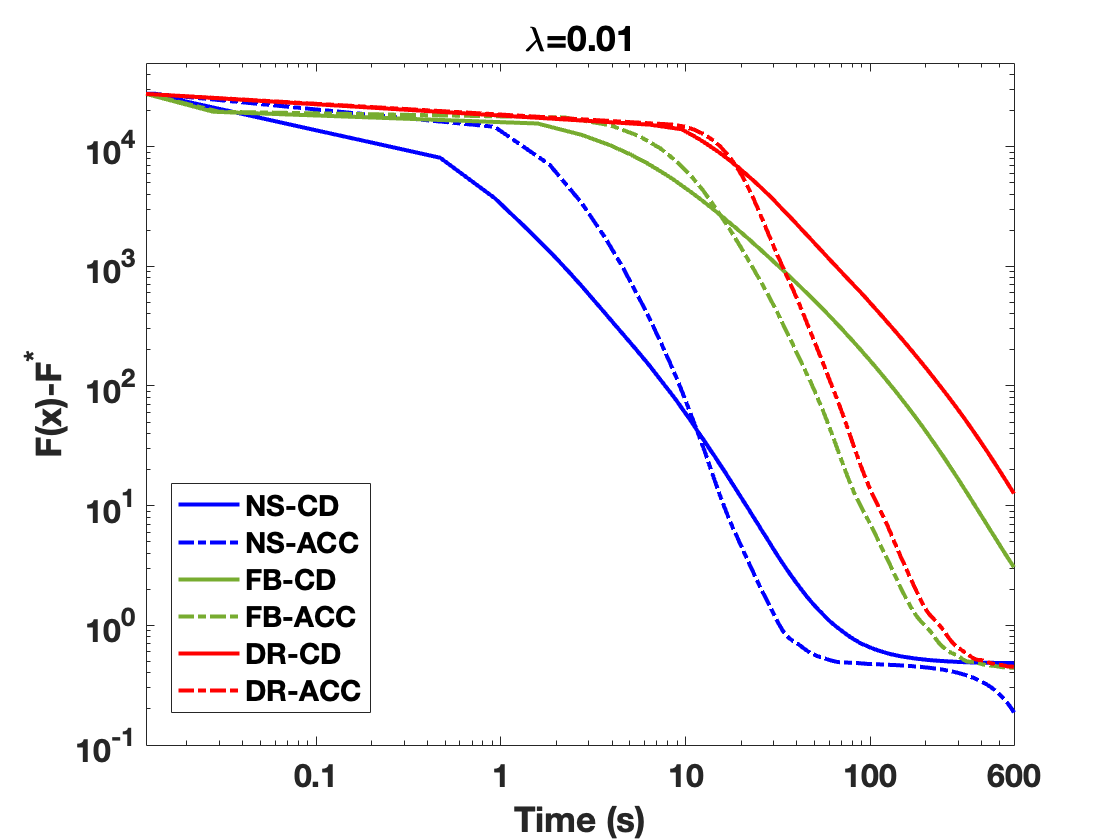} 
	\caption{Evolution of Algorithms 1 and 2 for four different smoothing techniques in  function values along time for quadratic objective with $\ell_2$ regularization and $n=10^4$.}
\end{figure}

\subsection{Quadratic objective  with TV regularization}
We also consider a quadratic problem with TV regularization:
\begin{equation}
\label{eq:TV}
  \min_{x \in \mathbb{R}^{n}}  \dfrac{1}{2} \|Bx - c \|^2 + \lambda \sum_{i=1}^{n-1} |x^{(i)} - x^{(i+1)}| 
  \vspace{-0.2cm}
\end{equation}
\noindent with $B \in \mathbb{R}^{n\times n}$ and $c\in\mathbb{R}^{n}$. Vectors  $c\in\mathbb{R}^{n}$ and $x_{0} \in \mathbb{R}^{n}$ are generated  from a normal distribution $\mathcal{N}(0,1)$ and the matrix $B \in \mathbb{R}^{n \times n}$ is generated as $B = Q^{T}CQ$, where $Q \in \mathbb{R}^{n \times n}$ is a sparse orthogonal matrix  and $C \in \mathbb{R}^{n \times n}$  is a diagonal matrix such that { $C=\text{diag} (100,\text{rand}(n/2-1,1),\text{zeros}(n/2,1))$.
Note that this function satisfies $q$-growth condition with $q=2$. Indeed, consider the function $g:\mathbb{R}^{n} \to \mathbb{R}$ such that $g(y) = \dfrac{1}{2}\|y - c \|$. We have that $g$ is strongly convex with strong convexity paramenter $\sigma_{g} = 1$ and the gradient of $g$ is Lipschitz with constant $L_{g}=1$. Moreover $\lambda \sum_{i=1}^{n-1} |x^{(i)} - x^{(i+1)}|$ is a polyhedral function. Hence, from \cite[Theorem 10]{NecNesGli:19}), the objective function in \eqref{eq:TV}
satisfies $q$-growth condition with $q=2$}. Thus, we can also apply Algorithm 3 for solving this problem. 
We consider the following methods: \\
1. NS-CD: Nesterov's smoothing using Algorithm 1. \\
2. NS-RT: Nesterov's smoothing using Algorithm 3. \\
3. FB-CD: Forward-Backward envelope using Algorithm 1. \\
4. FB-RT: Forward Backward evelope using Algorithm 3. \\
5. ME-CD: Moreau envelope using Algorithm 1. \\
6. ME-RT: Moreau evelope using Algorithm 3. \\
7. PROX: full proximal gradient method. \\
8. APPROX-RT: full restart accelerated algorithm proposed in \cite{FerQu:20}.

\noindent In all restart accelerated algorithms we update the restart periods as in Assumption \ref{assum}. Moreover, the first restart period is chosen $K_{0} = \lfloor0.01 e \cdot n \rfloor$.
We plot the function values along time (in seconds) in Figure 
\ref{figTV1}. 
As we can see from this figure, Algorithms 1 and 3 on Nesterov's smoothing are the fastest, with restart accelerated variant having the better performance.

\vspace{-0.4cm}

\begin{figure}[ht!]
	\centering
	\includegraphics[width=6.3cm,height=5cm]{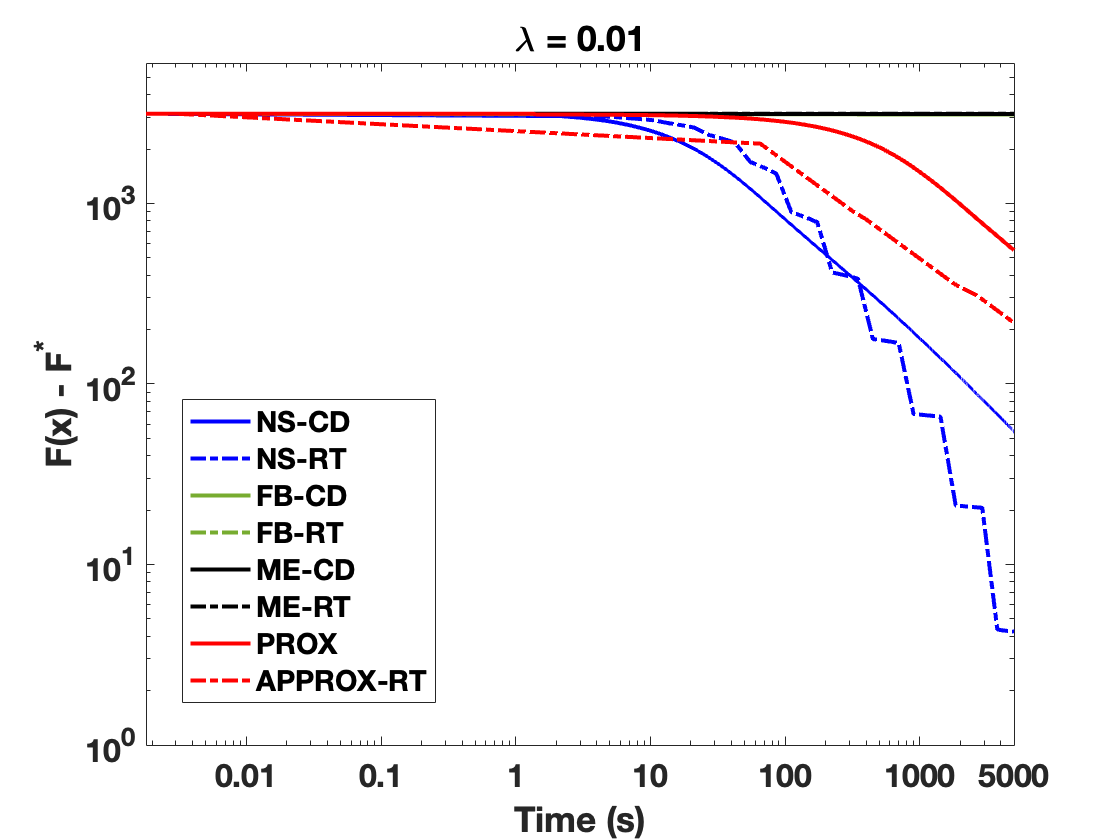} 
	\includegraphics[width=6.3cm,height=5cm]{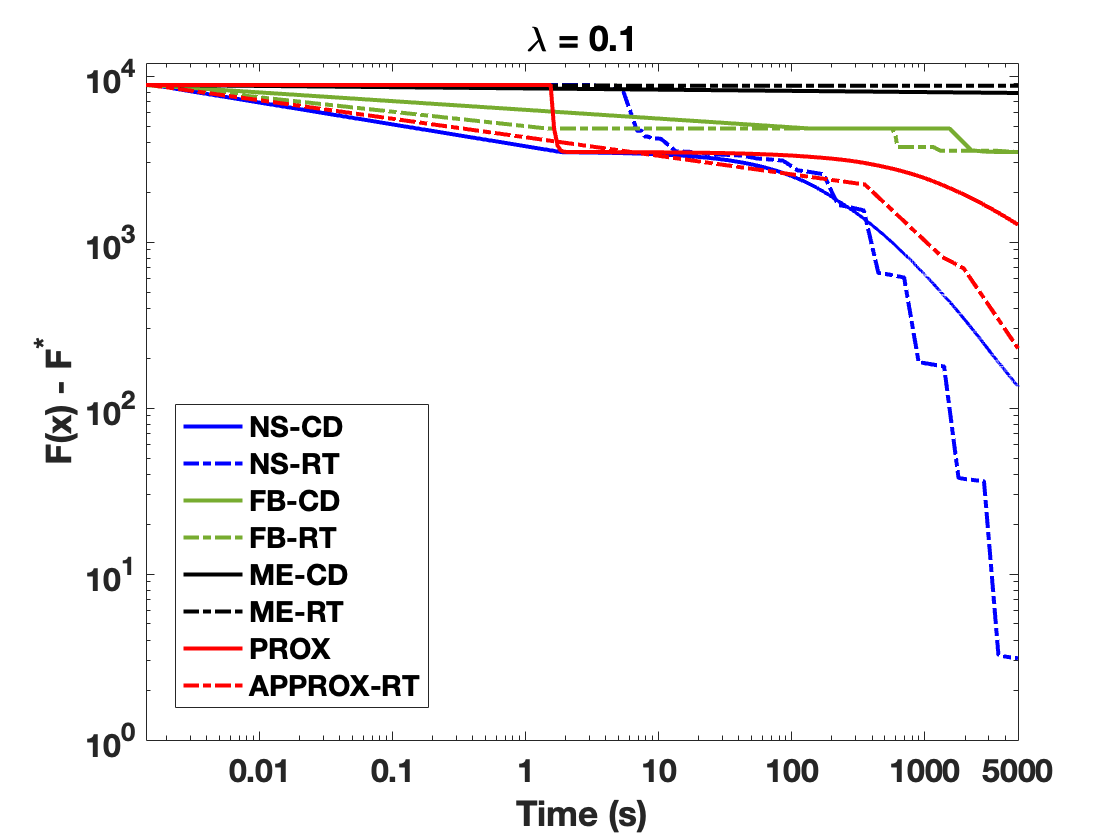} 
	\caption{ Evolution of Algorithms 1, 3 and their full variants for three different smoothing techniques in  function values along time for quadratic objective with TV regularization and $n=10^4$.}
 \label{figTV1}
\end{figure}


\vspace{-0.5cm}

\section{Conclusions}
\label{sec13}
In this paper we have considered  a general framework for smooth approximations of nonsmooth and nonseparable objective functions, which covers in particular Moreau envelope, Forward-Backward envelope, Douglas-Rachford envelope and Nesterov's smoothing. We have derived  convergence rates  for random (accelerated) coordinate descent methods for minimizing the smooth approximation. Moreover, under an additional $q$-growth assumption on the original function, we have derived improved rates (even linear) for  (restart accelerated) coordinate descent variants.  We  have also introduced a relative randomized coordinate descent algorithm when the original function is  relative smooth w.r.t. a (possibly nonseparable) differentiable function. Convergence rates have been also derived for this algorithm in the convex and $q$-growth cases. We have tested our numerical methods on two well-known applications and the numerical results show their efficiency.

\bmhead{Acknowledgments}
The research leading to these results has received funding from: TraDE-OPT funded by the European Union’s Horizon 2020 Research and Innovation Programme under the Marie Skłodowska-Curie grant agreement No.  861137;  UEFISCDI PN-III-P4-PCE-2021-0720, under project L2O-MOC, nr.  70/2022.

\vspace{-0.3cm}

\section*{Declarations}
\vspace*{-0.4cm}
There is no conflict of interest.

\begin{appendices}
\section{}\label{secA1}

\noindent  \red{  \textbf{Proof of Lemma 1:} Consider $x\in X$ and $h_{0} \in \mathbb{R}^{n_{i}}$, such that $x+U_{i}h_{0} \in X$, fixed. Define the function $r : \mathbb{R}^{n_{i}} \to \mathbb{R}$:
\begin{equation}
	\label{eq:r}
	r(h) = f(x+U_{i}h) - \langle U_{i}^{T}\nabla f(x+U_{i}h_{0}),h \rangle.
\end{equation} }
\red{
\noindent Since $f$ is convex along coordinates, it implies that $r$ is convex. Moreover,
\begin{equation*}
	 \nabla r(h) = U_{i}^{T}\nabla f(x+U_{i}h) - U_{i}^{T}\nabla f(x+U_{i}h_{0}).
\end{equation*} }
\red{
\noindent Hence $h^{*} = h_{0}$ is a optimal point of the function \eqref{eq:r}.  This implies that
\begin{equation*}
 r(h_{0}) \leq r\left(x - \dfrac{1}{L_{i}} U_{i} \nabla r(0)\right).   
\end{equation*} }
\red{
\noindent Combining the inequality above with \eqref{lip:3}, we obtain:
\begin{align*}
	&f(x+U_{i}h_{0}) - \langle U_{i}^{T}\nabla f(x+U_{i}h_{0}),h_{0} \rangle \\
	 &\leq f\left(x - \dfrac{1}{L_{i}} U_{i} \nabla r(0)\right)  + \dfrac{1}{L_{i}}  \langle U_{i}^{T}\nabla f(x+U_{i}h_{0}), \nabla r(0) \rangle \\ 
	 &\leq f(x)  + \dfrac{1}{L_{i}}  \langle U_{i}^{T}\nabla f(x+U_{i}h_{0}) - U_{i}^{T}\nabla f(x), \nabla r(0) \rangle + \dfrac{1}{2L_{i}} \|\nabla r(0)\|^{2} 	 \\
	 &= f(x)  - \dfrac{1}{L_{i}}  \| \nabla r(0) \|^2 + \dfrac{1}{2L_{i}} \|\nabla r(0)\|^{2} \\
	 &= f(x)  - \dfrac{1}{2L_{i}} \| \nabla r(0) \|^2. 
\end{align*}}
\red{
\noindent This implies that:
\begin{equation*}
	f(x+U_{i}h_{0}) - \langle U_{i}^{T}\nabla f(x+U_{i}h_{0}),h_{0} \rangle +  \dfrac{1}{2L_{i}}  \|  U_{i}^{T}\nabla f(x+U_{i}h_{0}) - U_{i}^{T}\nabla f(x)\|^{2} \leq f(x).
\end{equation*} }
\red{
\noindent Exchanging $x$ and  $x+U_{i}h_{0}$ in the inequality above and summing up, we obtain:
\begin{equation*}
	\dfrac{1}{L_{i}}  \|  U_{i}^{T}\nabla f(x+U_{i}h_{0}) - U_{i}^{T}\nabla f(x)\|^{2}  \leq \langle U_{i}^{T}\left( \nabla f(x+U_{i}h_{0}) - \nabla f(x)\right), h_{0} \rangle.
\end{equation*}}
\red{
\noindent Using the Cauchy-Schwarz inequality, we further  get:
\begin{equation*}
	\|  U_{i}^{T}\nabla f(x+U_{i}h_{0}) - U_{i}^{T}\nabla f(x)\| \leq L_{i} \|h_{0}\|.
\end{equation*} }
\red{
\noindent Hence, the statement follows. }
\end{appendices}

\bibliography{sn-bibliography}
\end{document}